\documentclass[11pt]{article}

\usepackage[a4paper,margin=0.9in]{geometry}
\usepackage{amsmath,amssymb,amsthm}
\usepackage{mathrsfs}
\usepackage{xcolor}
\usepackage[numbers]{natbib}

\numberwithin{equation}{section}

\newtheorem{theorem}{Theorem}[section]
\newtheorem{condition}[theorem]{Condition}
\newtheorem{lemma}[theorem]{Lemma}
\newtheorem{proposition}[theorem]{Proposition}
\newtheorem{corollary}[theorem]{Corollary}
\theoremstyle{definition}
\newtheorem{remark}[theorem]{Remark}
\newtheorem{example}[theorem]{Example}
\theoremstyle{plain}

\newcommand{\R}{\mathbb R}
\newcommand{\E}{\mathbb E}
\newcommand{\Pp}{\mathbb P}
\newcommand{\dd}{\mathrm d}

\title{Explosion Criteria for L\'evy-Driven SDEs and Applications to SPDEs}
\author{Pei-Sen Li\thanks{School of Mathematics and Statistics, Beijing Institute of Technology, Beijing, 100872, China; \texttt{peisenli@bit.edu.cn}}
\and Yuichi Shiozawa\thanks{Department of Mathematical Sciences, Faculty of Science and Engineering, Doshisha University, 1-3 Tatara Miyakodani, Kyotanabe, Kyoto 610-0394, Japan; \texttt{yshiozaw@mail.doshisha.ac.jp}}
\and Jian Wang\thanks{School of Mathematics and Statistics \& Key Laboratory of Analytical Mathematics and Applications (Ministry of Education) \& Fujian Provincial Key Laboratory of Statistics and Artificial Intelligence, Fujian Normal University, Fuzhou 350007, China; \texttt{jianwang@fjnu.edu.cn}}}
\date{}

\begin{document}
\maketitle
\begin{abstract}
We study finite-time explosion to $-\infty$ for the one-dimensional
stochastic differential equation
\[
    X_t^x=x-\int_0^t b(X_s^x)\,\dd s+L_t,\quad x\in\R,
\]
where $b$ is locally Lipschitz and $(L_t)_{t\ge0}$ is a pure-jump
L\'evy process.  We give verifiable conditions involving Osgood-type
integrability under which explosion
occurs almost surely from every finite initial state.  Under stronger
tail conditions, we obtain a uniform bound on the mean explosion time.
In particular, under  
suitable tail monotonicity, converse results make these criteria
necessary and sufficient, and yield an explicit classification for
Riccati drifts.
As an application, we consider a semilinear parabolic stochastic partial
differential equation (SPDE) with additive L\'evy space--time white noise
and homogeneous Dirichlet boundary conditions.  We give explicit conditions
on the nonlinearity and the L\'evy measure under which every local weak
solution defined up to explosion has an almost surely
finite explosion time.
\end{abstract}


\section{Introduction}
\subsection{Background}
Finite-time explosion (blow-up) is a fundamental phenomenon for stochastic
differential equations (SDEs), referring to sample trajectories leaving every
bounded set in finite random time.  For the ordinary differential equation (ODE) $\dot x=-b(x)$ with $b>0$, the
classical Osgood criterion characterizes finite-time explosion to $-\infty$ by
$\int_{-\infty} \dd u/b(u)<\infty$.  For one-dimensional
diffusions, Feller's classification and integral tests determine whether an
infinite boundary is accessible \cite{Feller54}.  We ask the corresponding
question when the continuous noise is replaced by additive pure-jump noise.
More precisely, we study
the explosion problem for the SDE 
\[
\dd X_t = -b(X_t)\,\dd t + \dd L_t,\quad X_0=x\in \R
\]
with a locally Lipschitz drift $b$ and a pure-jump L\'evy process
$(L_t)_{t\ge0}$.  Motivated by the stochastic partial differential equation (SPDE) application below, we study the
accessibility of the left boundary $-\infty$.

Denote by \(T_-^x\) the explosion time to \(-\infty\) from the initial state \(x\).
We give explicit conditions on \(b\) and \(\nu\) under which the process
explodes to \(-\infty\) almost surely from every finite initial state:
\[
    \Pp(T_-^x<\infty)=1,\quad x\in\R.
\]
The assumptions separate two parts of the explosion mechanism.  The
nondegeneracy and left-tail Osgood conditions yield a uniform positive lower
bound on the probability of explosion within a fixed time over all initial
states below any prescribed level.  The right-tail return condition allows
this estimate to be restarted on survival, and the strong Markov property then
yields almost-sure explosion.  Under the same nondegeneracy and left-tail
assumptions, the right-tail Osgood condition yields a uniform bound on the mean
time required either to enter a fixed compact interval or to explode.  Together
with the renewal step and the Markov property, this gives
\[
    \sup_{x\in\R}\E T_-^x<\infty.
\]
This stopping-time iteration is
analogous to the cycle argument used by Chow and
Khasminskii~\cite{ChowKhasminskii14} for diffusions.  Establishing the required
uniform estimates for a pure-jump equation requires
controlling overshoots and the compensated small-jump drift.
We emphasize that, the counterexample in Example~\ref{ex:b-irreducibility-counterexample} shows
that
topological irreducibility does not by itself provide the required right-tail
control.

We also identify when the tail conditions are necessary.  Under tail
monotonicity, two pathwise comparison estimates distinguish the roles of the two ends: divergence of
the left Osgood integral prevents left explosion, whereas divergence of the right
Osgood 
integral prevents uniform descent from arbitrarily large positive
initial states.  These estimates yield separate
necessary and sufficient criteria for almost-sure explosion and for
uniform boundedness of the mean explosion time.
The Riccati equation with drift $b(x)=x^2+a$, $a\in\R$, gives a concrete
classification, including the initial-state dependence when the access
condition fails.

Several explosion criteria are known for special classes of one-dimensional
jump processes.  For stable jump diffusions of the form
\(\dd Z_t=\sigma(Z_{t-})\,\dd S_t\), D\"oring and
Kyprianou~\cite{DoringKyprianou20} obtained necessary and sufficient integral
criteria for explosion; see also Baguley, D\"oring and
Kyprianou~\cite{BaguleyDoringKyprianou24} and Baguley, D\"oring and
Shi~\cite{BaguleyDoringShi26} for related results on stable or time-changed
L\'evy models.  For a class of SDEs with nonnegative jumps, Li, Yang and
Zhou~\cite{LiYangZhou19} gave criteria for non-explosion and
positive-probability explosion.
To our knowledge, for the additive pure-jump equation considered here, the
existing literature does not provide explicit criteria for almost-sure
explosion from every finite initial state or for a uniform bound on the mean
explosion time.  We establish both conclusions under conditions on the drift
and the L\'evy measure.

The preceding SDE result also applies to stochastic reaction--diffusion
equations.
For the equation driven by additive Gaussian space--time white noise,
Fern\'andez Bonder and Groisman~\cite{BG} used the first-eigenfunction
projection and Feller's test to prove that the Osgood condition is sufficient
for almost-sure finite-time blow-up.  Foondun and
Nualart~\cite{FoondunNualart21} later proved the converse under suitable
assumptions, obtaining an Osgood characterization on a bounded interval.
In a different direction, Shiozawa and Wang~\cite{SW} obtained explicit
integral tests for the spatial asymptotics of linear fractional stochastic
heat equations driven by additive L\'evy white noise.

For additive pure-jump space--time white noise, we test the weak equation
against the normalized positive first Dirichlet eigenfunction and apply
Jensen's inequality to the convex reaction term.  The projected noise is a
one-dimensional pure-jump L\'evy process, and the resulting pathwise comparison
places the weighted projection of the SPDE solution above the solution of a
L\'evy-driven SDE.  Our SDE criterion therefore implies that every local weak
solution has an almost surely finite explosion time.  This conclusion
differs from the mean \(L^p\)-norm blow-up criteria of Bao and
Yuan~\cite{BaoYuan16} for equations with state-dependent Gaussian and jump
coefficients and of Li, Peng and Jia~\cite{LiPengJia17} for positive solutions
with multiplicative L\'evy-type noise.

\subsection{Main result: Finite explosion for one-dimensional L\'evy driven SDEs}
Let \(\nu\) be a L\'evy measure on \(\R\setminus\{0\}\) such that
\[
    \nu(\R\setminus\{0\})>0,
    \quad
    \int_{\R\setminus\{0\}}(1\wedge z^2)\,\nu(\dd z)<\infty,
\]
and let
\(\mathcal N(\dd s,\dd z)\) be a Poisson random measure on
\([0,\infty)\times(\R\setminus\{0\})\) with intensity
\(\dd s\,\nu(\dd z)\).  Write
\(\widetilde{\mathcal N}(\dd s,\dd z)
=\mathcal N(\dd s,\dd z)-\dd s\,\nu(\dd z)\) on \(\{|z|\le1\}\)
for the compensated measure of \(\mathcal N(\dd s,\dd z)\).
For any $t\ge0$, let
\[
    L_t
    :=
    \int_0^t\int_{\{|z|\le1\}}z\,\widetilde{\mathcal N}(\dd s,\dd z)
    +
    \int_0^t\int_{\{|z|>1\}}z\,\mathcal N(\dd s,\dd z).
\]
Let \((\mathcal F_t)_{t\ge0}\) be the usual natural filtration generated
by \(\mathcal N\).

Throughout the paper, we assume that \(b:\R\to\R\) is locally
Lipschitz.  For each \(x\in\R\), consider the integral equation before the
explosion time:
\begin{equation}\label{eq:b-finite-line-equation}
    X_t^x
    =
    x-\int_0^t b(X_s^x)\,\dd s+L_t.
\end{equation}

For \(a\in\R\), define its downward and upward hitting times respectively by
\[
    \tau_{a,-}^x:=\inf\{t\ge0:X_t^x\le a\},
    \quad
    \tau_{a,+}^x:=\inf\{t\ge0:X_t^x\ge a\}
\]
with the convention \(\inf\emptyset=\infty\), and define the left and right
explosion times respectively by
\[
    T_-^x:=\lim_{n\to\infty}\tau_{-n,-}^x,
    \quad
    T_+^x:=\lim_{n\to\infty}\tau_{n,+}^x .
\]
The explosion time of \(X^x\) is \(T_-^x\wedge T_+^x\).
Since \(b\) is locally Lipschitz, \eqref{eq:b-finite-line-equation} has a unique
strong c\`adl\`ag solution \(X^x_t\) for \(0\le t<T_-^x\wedge T_+^x\).  At \(T_-^x\)
(respectively, \(T_+^x\)) the process  \(X^x_t\) is sent to \(-\infty\) (respectively,
\(+\infty\)), which is then absorbing.
With this convention, \(X^x:=(X_t^x)_{t\ge0}\), taking values in
\(\R\cup\{-\infty,+\infty\}\), is a time-homogeneous strong Markov process
with respect to \((\mathcal F_t)_{t\ge0}\).
We refer to
\cite[Theorem 38, p.~309]{P} 
for the existence and uniqueness
of solutions to \eqref{eq:b-finite-line-equation} up to
\(T_-^x\wedge T_+^x\).

We first state the three conditions used in our general results.

\begin{condition}[Return from the right tail]
\label{cond:b-right-return}
There exists \(R\in\R\) such that
\[
    \Pp(\tau_{R,-}^x<\infty)=1\quad \text{for every }x>R .
\]
\end{condition}

\begin{condition}[Compact access and left-tail explosion]
\label{cond:b-as-coeff}
Set
\[
    \kappa_+:=\int_{(0,1]} z\,\nu(\dd z)\in[0,\infty]
\]
with the convention that \(b(x)+\kappa_+=\infty\) if \(\kappa_+=\infty\).
The following conditions hold.
\begin{enumerate}
\renewcommand{\labelenumi}{\textup{(\roman{enumi})}}
\item
\[
    b(x)+\kappa_+>0\quad\text{for every }x\in\R
    \quad\text{or}\quad
    \nu((-\infty,0))>0 .
\]
\item There exist \(a_-\in\R\) and a nonincreasing
function
\(g_-:(-\infty,a_-]\to(0,\infty)\) such that
\[
    b(x)\ge g_-(x),\quad x\le a_-,
\]
and
\[
    \int_{-\infty}^{a_-}\frac{\dd u}{g_-(u)}<\infty .
\]
\end{enumerate}
\end{condition}

\begin{condition}[Right-tail Osgood condition]
\label{cond:b-lyapunov}
There exist \(a_+\in\R\) and a nondecreasing
function
\(g_+:[a_+,\infty)\to(0,\infty)\) such that
\[
    b(x)\ge g_+(x),\quad x\ge a_+,
\]
and
\[
    \int_{a_+}^\infty \frac{\dd u}{g_+(u)}<\infty;
\]
\end{condition}

The main results in this paper are as follows.
\begin{theorem}[Almost sure explosion]\label{thm:b-as-explosion}
Assume that Conditions~$\ref{cond:b-right-return}$ and $\ref{cond:b-as-coeff}$ hold.  Then
\begin{equation}\label{eq:b-as-explosion}
    \Pp(T_-^x<\infty)=1,\quad x\in\R .
\end{equation}
\end{theorem}

\begin{theorem}[Finite mean explosion]\label{thm:b-main}
Assume that Conditions~$\ref{cond:b-as-coeff}$ and $\ref{cond:b-lyapunov}$ hold.  Then
\begin{equation}\label{eq:b-finite-mean-explosion}
    \sup_{x\in\R}\E T_-^x<\infty .
\end{equation}
\end{theorem}

We make a few comments on Conditions~\ref{cond:b-right-return},
\ref{cond:b-as-coeff} and \ref{cond:b-lyapunov}.

\begin{remark}
\label{rem:b-assumption-roles}
\begin{itemize}
\item[{\rm(i)}] Condition~\ref{cond:b-right-return} guarantees that, from every
sufficiently large starting point, the process $X^x$ goes below to a lower level
almost surely.  It is necessary for \eqref{eq:b-as-explosion}: if
\(T_-^x<\infty\) almost surely for all \(x\), then, for every finite \(R\) and
every \(x>R\), \(\tau_{R,-}^x<\infty\) almost surely.
Lemma~\ref{lem:b-right-lyapunov-return} below gives a verifiable Lyapunov
sufficient condition for Condition~\ref{cond:b-right-return}.
\item[{\rm(ii)}] Condition~\ref{cond:b-as-coeff}(i) is also necessary for
\eqref{eq:b-as-explosion}.  If it failed, then
\(\nu((-\infty,0))=0\), \(\kappa_+<\infty\), and
\(b(x_0)+\kappa_+\le0\) for some \(x_0\in\R\).  In that case,
\[
    X_t^x
    =
    x-\int_0^t\bigl(b(X_s^x)+\kappa_+\bigr)\,\dd s+J_t,
    \quad 0\le t<T_-^x\wedge T_+^x,
\]
where \(J_t:=L_t+\kappa_+t\), \(t\ge0\), is nondecreasing.
Since $-b(x_0)-\kappa_+\ge0$ and all jumps are nonnegative,
solutions starting at or above $x_0$ cannot cross below $x_0$.
Thus, for every $x\ge x_0$,
\[
    X_t^x\ge x_0,\quad 0\le t<T_-^x\wedge T_+^x.
\]
Consequently, \(\Pp(T_-^x<\infty)=0\) for every \(x\ge x_0\).

\item[{\rm(iii)}] Condition~\ref{cond:b-as-coeff}(ii) supplies the Osgood mechanism after the
process $X^x$ has reached the far-left region.  Together with
Condition~\ref{cond:b-as-coeff}(i), it yields the uniform positive explosion
probability in Proposition~\ref{prop:b-halfline-explosion}; the almost-sure
conclusion additionally uses Condition~\ref{cond:b-right-return} and the renewal
step.
\item[{\rm(iv)}] Together with Condition~\ref{cond:b-as-coeff}(ii),
Condition~\ref{cond:b-lyapunov} gives a uniform mean bound for reaching a fixed
finite level from arbitrarily large positive states.  This is the extra
condition needed for \eqref{eq:b-finite-mean-explosion}, and
Proposition~\ref{cor:b-right-return-from-osgood} below shows that the same two conditions
imply Condition~\ref{cond:b-right-return}.
\end{itemize}
\end{remark}

As shown in Remark~\ref{rem:b-assumption-roles}(i),
Condition~\ref{cond:b-right-return} is necessary for
\eqref{eq:b-as-explosion}.  The following example shows that, even under
Condition~\ref{cond:b-as-coeff}, topological irreducibility cannot replace it.
\begin{example}[Irreducibility does not replace right-tail return]
\label{ex:b-irreducibility-counterexample}
Fix \(\alpha\in(1,2)\) and \(c_\alpha,\gamma>0\).  Let \((L_t)_{t\ge0}\) be a symmetric
\(\alpha\)-stable process with L\'evy measure as follows
\[
    \nu(\dd z)=c_\alpha |z|^{-1-\alpha}\,\dd z,
    \quad z\in\R\setminus\{0\},
\]
and choose \(b\in C^\infty(\R)\) such that
\[
    b(x)=x^2,\,\, x\le-1,
    \quad
    b(x)=-\gamma,\,\,x\ge0.
\]
Then Condition~\ref{cond:b-as-coeff} holds with
\(a_-=-1\) and \(g_-(x)=x^2\).  For a nonempty open set \(U\subset\R\), set
\[
    \tau_U^x:=\inf\{t\ge0:X_t^x\in U\}.
\]
For every \(x\in\R\) and every nonempty open set \(U\subset\R\),
the full support
property
of the symmetric stable noise gives some \(t_0>0\) such that
\[
    \Pp(\tau_U^x\le t_0)>0.
\]
Thus \(X^x\) is topologically irreducible on \(\R\).

Since \(\alpha>1\), \(\E|L_1|<\infty\); and the symmetry gives \(\E L_1=0\).
The strong law of large numbers for
the symmetric
\(\alpha\)-stable process $(L_t)_{t\ge0}$ therefore yields
\[
    \frac{L_t}{t}\rightarrow0,
   \quad
    \frac{\gamma t+L_t}{t}\rightarrow\gamma>0,
    \quad
    \inf_{t\ge0}(\gamma t+L_t)>-\infty
    \,\,\textrm{ almost surely}.
\]
For \(x>0\), define
\[
    A_x:=\left\{\inf_{t\ge0}(x+\gamma t+L_t)>0\right\}.
\]
Then
\[
    \lim_{x\to\infty}\Pp(A_x)=1.
\]
In particular, \(\Pp(A_x)>0\) for every sufficiently large \(x\).  On \(A_x\),
\(x+\gamma s+L_s>0\) for every \(s\ge0\), and hence
\[
    b(x+\gamma s+L_s)=-\gamma,
    \quad s\ge0.
\]
Therefore, for every \(t\ge0\),
\[
    x-\int_0^t b(x+\gamma s+L_s)\,\dd s+L_t
    =x+\gamma t+L_t
    \quad\text{on }A_x.
\]
Thus \(t\mapsto x+\gamma t+L_t\) is a global solution of
\eqref{eq:b-finite-line-equation} on \(A_x\).  Pathwise uniqueness gives
\[
    X_t^x=x+\gamma t+L_t>0,
    \quad t\ge0,
\]
and consequently
\[
    \Pp(T_-^x=\infty)
    \ge
    \Pp(A_x)
    >0
\]
for all sufficiently large \(x\).  Thus, the topological irreducibility, even
together with Condition~\ref{cond:b-as-coeff}, does not imply
\eqref{eq:b-as-explosion}; Condition~\ref{cond:b-right-return}, or a stronger
recurrence assumption, is still needed.
\end{example}

\subsection{Sharp criteria and Riccati equations}
\label{subsec:b-sharp-main}

Under suitable tail monotonicity assumptions on $b$, we obtain the
following necessary and sufficient conditions.

\begin{theorem}[Osgood criteria under tail monotonicity]
\label{thm:b-sharp-osgood}
Suppose that $b$ is positive and nonincreasing on $(-\infty,a_-]$
for some $a_-\in\R$. Set
\[
    I_-:=\int_{-\infty}^{a_-}\frac{\dd u}{b(u)}.
\]
\begin{enumerate}
\renewcommand{\labelenumi}{\textup{(\roman{enumi})}}
\item
\[
    \Pp(T_-^x<\infty)=1\quad\text{for every }x\in\R
\] holds true
if and only if
Conditions~$\ref{cond:b-right-return}$ and~$\ref{cond:b-as-coeff}${\rm(i)} hold
and $I_-<\infty$.
\item Suppose in addition that $b$ is positive and nondecreasing on
$[a_+,\infty)$ for some $a_+>a_-$. Set
\[
    I_+:=\int_{a_+}^{\infty}\frac{\dd u}{b(u)}.
\]
Then
\[
    \sup_{x\in\R}\E T_-^x<\infty
\]
if and only if
Condition~$\ref{cond:b-as-coeff}${\rm(i)} holds and $I_-+I_+<\infty$.
\end{enumerate}
\end{theorem}

\begin{corollary}[Riccati classification]
\label{cor:b-riccati-classification}
In the SDE~\eqref{eq:b-finite-line-equation}, let $b(x)=x^2+a$, where
$a\in\R$.
\begin{enumerate}
\renewcommand{\labelenumi}{\textup{(\roman{enumi})}}
\item If
\begin{equation}\label{eq:b-riccati-access}
    \nu((-\infty,0))>0\quad\text{or}\quad a+\kappa_+>0,
\end{equation}
then
\[
    \sup_{x\in\R}\E T_-^x<\infty.
\]
In particular, $\Pp(T_-^x<\infty)=1$ for every $x\in\R$.
\item If \eqref{eq:b-riccati-access} fails, i.e., if
$\nu((-\infty,0))=0$ and $a+\kappa_+\le0$, then
\begin{equation}\label{eq:b-riccati-failure-regimes}
    \Pp(T_-^x<\infty)=0\,\, \hbox{ for all } \,\,x\ge-r; \quad   0<\Pp(T_-^x<\infty)<1\,\, \hbox{ for all } \,\,x<-r,
\end{equation}where
$r=\sqrt{-(a+\kappa_+)}$.
\end{enumerate}
\end{corollary}

\begingroup
\def\d{\mathrm d}

\subsection{Application to an SPDE with additive L\'evy noise}
We apply Theorem~\ref{thm:b-as-explosion} to the parabolic SPDE
\begin{equation}\label{e:SPDE}
    u_t=u_{xx}+f(u)+\sigma\Lambda(\d x,\d t)
\end{equation}
on \((0,1)\), with initial condition \(u(0)=u_0\) and homogeneous
Dirichlet boundary conditions.  Here \(\sigma>0\),
\(f:\R\to\R\) is locally Lipschitz, and \(\Lambda\) is a L\'evy
space--time white noise.  For this application, let
\((\Omega,\mathcal F,(\mathcal G_t)_{t\ge0},\Pp)\) be a filtered probability
space satisfying the usual conditions.  Let \(\lambda\) be a
nontrivial L\'evy measure on
\(\R\setminus\{0\}\), and let \(\mu\) be a Poisson random measure relative
to \((\mathcal G_t)_{t\ge0}\) on
$
    (0,1)\times(\R\setminus\{0\})\times(0,\infty)
$
with intensity
\(\widehat\mu(\d x,\d z,\d t):=\d x\,\lambda(\d z)\,\d t\), and set
\[
\Lambda(\d x,\d t)
=\int_{\{|z|\le1\}}z(\mu-\widehat\mu)(\d x,\d z,\d t)
+\int_{\{|z|>1\}}z\,\mu(\d x,\d z,\d t).
\]

We first specify the state space and the precise meaning of a solution to
\eqref{e:SPDE}.  Let \(H^1(0,1)\) be the Sobolev space of functions in
\(L^2(0,1)\) whose weak derivative also belongs to \(L^2(0,1)\), let
\(H_0^1(0,1)\) be the closure of \(C_c^\infty(0,1)\) in \(H^1(0,1)\), and set
\[
    H^{-1}(0,1):=(H_0^1(0,1))^*.
\]
We write \(\langle\cdot,\cdot\rangle\) for the duality pairing between
\(H^{-1}(0,1)\) and \(H_0^1(0,1)\).  For \(1\le q\le\infty\),
\(\|\cdot\|_q\) denotes the \(L^q(0,1)\) norm.
We use the usual Sobolev norm
\(\|\psi\|_{H^1}:=(\|\psi\|_2^2+\|\psi'\|_2^2)^{1/2}\)
on \(H_0^1(0,1)\) and the corresponding dual norm \(\|\psi\|_{H^{-1}(0,1)}\) on \(H^{-1}(0,1)\).
The inequality
\(\|\psi\|_\infty\le \|\psi'\|_2\), valid for every
\(\psi\in H_0^1(0,1)\), shows that
\(L^1(0,1)\) functions and finite signed Radon measures embed continuously
into \(H^{-1}(0,1)\); in particular,
\(\langle\delta_x,\psi\rangle=\psi(x)\) for \(x\in(0,1)\).

This state-space choice is consistent with
\cite[Theorem~2.5 and Remark~2.6]{ChongDalangHumeau19}, where c\`adl\`ag paths
are obtained in \(H^r\) for every \(r<-1/2\) for stochastic heat equations
with a globally Lipschitz noise coefficient driven by L\'evy noise.

\medskip
\noindent\textbf{Local weak solutions up to explosion.}
Fix a deterministic \(u_0\in L^2(0,1)\).  By a local weak solution
up to explosion on the stochastic interval \([0,T)\)
we mean a pair \((u,T)\) with the following properties.
Here \(T:\Omega\to(0,\infty]\) is a stopping time, and
\(u=(u(t))_{0\le t<T}\) is an adapted \(H^{-1}(0,1)\)-valued process with
c\`adl\`ag paths and \(u(0)=u_0\).
We further require that, almost surely on $\{T<\infty\}$,
\begin{equation}\label{eq:spde-norm-explosion}
    \sup_{0\le t<T}\|u(t)\|_{H^{-1}(0,1)}=+\infty.
\end{equation}
We call $T$ the \emph{explosion time}.

Let \(\mathscr P_{\mathcal G}\) denote the
predictable \(\sigma\)-field on \(\Omega\times[0,\infty)\) associated with
\((\mathcal G_t)_{t\ge0}\).  There exists a real-valued
\(\mathscr P_{\mathcal G}\otimes\mathcal B((0,1))\)-measurable function
\[
    u:\Omega\times[0,\infty)\times(0,1)\longrightarrow\R,
    \quad (\omega,s,x)\longmapsto u(\omega,s,x),
\]
such that, for
\(\Pp\)-almost every \(\omega\) and Lebesgue-almost every
\(s<T(\omega)\), \(u(\omega,s,\cdot)\in L^1(0,1)\) and
\[
    \langle u(s),\psi\rangle
    =\int_0^1u(\omega,s,x)\psi(x)\,\d x
    \quad\text{for every }\psi\in H_0^1(0,1).
\]
The pair also satisfies, almost surely,
\[
    \int_0^t\int_0^1
    \bigl(|u(s,x)|+|f(u(s,x))|\bigr)\,\d x\,\d s<\infty,
    \quad 0\le t<T,
\]
and, for any \(\psi\in C^\infty([0,1])\) satisfying
\(\psi(0)=\psi(1)=0\), the following identity holds almost surely,
simultaneously for all \(0\le t<T\):
\begin{align*}
\langle u(t),\psi\rangle-\langle u_0,\psi\rangle
={}&\int_0^t\int_0^1u(s,x)\psi''(x)\,\d x\,\d s\\
&+\int_0^t\int_0^1 f(u(s,x))\psi(x)\,\d x\,\d s
+\sigma\int_0^t\int_0^1\psi(x)\,\Lambda(\d x,\d s).
\end{align*}

Corollary~\ref{cor:spde-poisson-maximal} below establishes
the existence of such solutions for a class of Poisson-driven equations.

\begin{theorem}\label{cor:spde-explicit-correction}
Let \(f\) be a nonnegative convex function such that
\[
    \int^\infty\frac{\d r}{f(r)}<\infty,
\]
and assume that \(\lambda((0,\infty))>0\).  Suppose, in addition, that at
least one of the following conditions holds:
\begin{enumerate}
\renewcommand{\labelenumi}{\textup{(\roman{enumi})}}
\item
\[
    \int_{(-\infty,-1)}\log(1+|y|)\,\lambda(\d y)<\infty;
\]
\item There exists $s_0>0$ such that $f(r)>0$ for any $r\le -s_0$ and
\[
    \int_{-\infty}\frac{\d r}{f(r)}<\infty.
\]
\end{enumerate}
Then, for every deterministic \(u_0\in L^2(0,1)\), every local weak
solution \((u,T)\) up to explosion satisfies
\[
    \Pp(T<\infty)=1 .
\]
If \textup{(ii)} holds, there exists $C<\infty$, depending only on
$f,\sigma,\lambda$, such that
\begin{equation}\label{eq:spde-general-mean}
    \sup_{u_0\in L^2(0,1)}\ \sup_{(u,T)}\E T\le C,
\end{equation}
where the inner supremum is over all such solutions starting from
$u_0$.
\end{theorem}

\begin{corollary}
\label{cor:spde-poisson-maximal}
Let $f(r)=|r|^p$ with $1<p<2$, and let $\lambda$ be a finite
measure on $\R\setminus\{0\}$ such that $\lambda((0,\infty))>0$.
For every deterministic
$u_0\in L^2(0,1)$, the SPDE~\eqref{e:SPDE} has a
local weak solution $(u,T)$ up to explosion.
Moreover, for some $C<\infty$ depending only on
$p,\sigma,\lambda$,
\begin{equation}\label{eq:spde-poisson-mean}
    \sup_{u_0\in L^2(0,1)}
        \E T\le C.
\end{equation}
In particular, $T<\infty$ almost surely.
\end{corollary}

\begin{remark}
Without noise (i.e., $\sigma=0$ in \eqref{e:SPDE}), the Dirichlet equation
$\partial_t v=\partial_{xx}v+|v|^p$ with initial value
$v(0,x)=a\sin(\pi x)$ has a global solution for sufficiently small $a>0$,
but blows up in finite time for sufficiently large $a$; see the discussion
in Fern\'andez Bonder and Groisman~\cite[Introduction]{BG}.
In contrast, Corollary~\ref{cor:spde-poisson-maximal} shows that adding
the Poisson noise considered there leads to almost-sure finite-time
explosion from every initial state $u_0\in L^2(0,1)$.
\end{remark}

\bigskip

The remainder of the paper is organized as follows.
Section~\ref{sec:b-lyapunov-input} develops Lyapunov criteria for the right-tail
return condition and verifies that condition under the stronger tail
assumptions, while Section~\ref{sec:b-compact-input} establishes a uniform
fixed-time lower bound on the probability of explosion below any prescribed
level.
Section~\ref{sec:b-renewal} combines these ingredients through stopping-time
renewal arguments to prove Theorems~\ref{thm:b-as-explosion}
and~\ref{thm:b-main}.
Section~\ref{sec:b-converse} proves the converse criteria and the Riccati
classification.  The final section proves
Theorem~\ref{cor:spde-explicit-correction} by projecting the SPDE onto the first
Dirichlet eigenfunction and comparing the projection with a one-dimensional
L\'evy-driven SDE, and constructs the local weak solutions
up to explosion in
Corollary~\ref{cor:spde-poisson-maximal}.

\section{Verifying the Tail Conditions}\label{sec:b-lyapunov-input}

\subsection{Condition~$\ref{cond:b-right-return}$}
\begin{lemma}\label{lem:b-no-right-explosion-from-return}
Assume that Condition~$\ref{cond:b-right-return}$ holds, or suppose that there exists
\(a\in\R\) such that
$b(y)\ge0$ for all $y\ge a .$
Then,
\[
    \Pp(T_+^x=\infty)=1,\quad x\in\R .
\]
Under the latter assumption, $T_+^x=\infty$ also holds,
when $(L_t)_{t\ge0}$ is replaced by any c\`adl\`ag function
$\ell:[0,\infty)\to\R$ with $\ell_0=0$.
\end{lemma}

\begin{proof}
\textup{(i)} Suppose that Condition~\ref{cond:b-right-return} holds, and let
\(R\) be the level appearing there.  Fix \(x\in\R\).  For each positive
rational number \(q\), set
\[
    A_q
    :=
    \left\{
        q<T_-^x\wedge T_+^x,\quad X_q^x>R,\quad
        X_t^x>R\ \text{for every }t\ge q
    \right\}.
\]
Since \(q\) is deterministic, the Markov property at time \(q\) and
Condition~\ref{cond:b-right-return} give
\[
\begin{aligned}
    \Pp(A_q)
    &=
    \E\!\left[
        \mathbf 1_{\{q<T_-^x\wedge T_+^x,\,X_q^x>R\}}
        \Pp\bigl(\tau_{R,-}^{X_q^x}=\infty\bigr)
    \right]
    =0 .
\end{aligned}
\]
On \(\{T_+^x<\infty\}\), one has
\(X_t^x\to+\infty\) as \(t\uparrow T_+^x\).  Since \(+\infty\) is
absorbing, for any fixed \(\omega\in\{T_+^x<\infty\}\) there exists
\(s(\omega)\in\mathbb Q\), \(0<s(\omega)<T_+^x(\omega)\), such that
\(X_t^x(\omega)>R\) for every \(t\ge s(\omega)\).  Hence
\(\omega\in A_{s(\omega)}\), and consequently
\[
    \{T_+^x<\infty\}
    \subseteq
    \bigcup_{\substack{q\in\mathbb Q\\ q>0}}A_q,
\]
and the right-hand side has probability zero.

\medskip
\noindent
\textup{(ii)} Suppose that \(b(y)\ge0\) for \(y\ge a\).  Fix \(x\in\R\),
any c\`adl\`ag driving path \((L_t)_{t\ge0}\) with \(L_0=0\),
and \(t_0<\infty\).  Since \(L\) is bounded on
\([0,t_0]\), set
\[
    C:=\sup_{0\le s\le t_0}|L_s|,
    \quad
    M:=\max\{x,a+C\}.
\]
For \(s<T_-^x\wedge T_+^x\), define \(Y_s^x:=X_s^x-L_s\).  By
\eqref{eq:b-finite-line-equation},
\[
    Y_s^x
    =
    x-\int_0^s b(Y_u^x+L_u)\,\dd u .
\]
Thus, \(s\mapsto Y_s^x\) is absolutely continuous, and
\(\dot Y_s^x=-b(Y_s^x+L_s)\) almost everywhere.  If \(0\le s\le t_0\) and
\(Y_s^x>M\), then
\[
     Y_s^x+L_s>M-C\ge a,
\]
and consequently \(\dot Y_s^x\le0\).  Since \(Y_0^x=x\le M\), for
\(0\le r<t_0\wedge T_-^x\wedge T_+^x\),
\[
\begin{aligned}
    (Y_r^x-M)^+
    &=
    \int_0^r
    \mathbf 1_{\{Y_s^x>M\}}\dot Y_s^x\,\dd s
    =
    -\int_0^r
    \mathbf 1_{\{Y_s^x>M\}}b(Y_s^x+L_s)\,\dd s
    \le0 .
\end{aligned}
\]
Since \((Y_r^x-M)^+\ge0\), the last display gives
\((Y_r^x-M)^+=0\).  Thus \(Y_r^x>M\) would be a contradiction, and hence
\(Y_r^x\le M\).  Therefore
\[
    X_r^x=Y_r^x+L_r\le M+C,
    \quad 0\le r<t_0\wedge T_-^x\wedge T_+^x .
\]
This rules out the right explosion before \(t_0\).  Since \(t_0<\infty\) was
arbitrary, \(T_+^x=\infty\).
\end{proof}

\begin{lemma}[Exit from a bounded interval]
\label{lem:b-no-bounded-trapping}
For every \(R<n\) and \(x\in(R,n)\),
\[
    \Pp\bigl(\tau_{R,-}^x\wedge\tau_{n,+}^x<\infty\bigr)=1 .
\]
\end{lemma}

\begin{proof}
Fix \(R<n\).  It is enough to prove that, for some \(p>0\),
\begin{equation}\label{eq:b-bounded-exit-unit}
    \inf_{R<x<n}
    \Pp\bigl(\tau_{R,-}^x\wedge\tau_{n,+}^x\le1\bigr)\ge p.
\end{equation} Indeed, applying the Markov property
inductively
and using
\eqref{eq:b-bounded-exit-unit} give
\[
    \sup_{R<x<n}
    \Pp\bigl(\tau_{R,-}^x\wedge\tau_{n,+}^x>k\bigr)
    \le(1-p)^k .
\]
Letting \(k\to\infty\) proves the desired claim.

By the non-zero property of the L\'evy measure $\nu$, there exist \(\varepsilon\in\{-1,1\}\) and
\(0<\ell<r<\infty\) such that \(\nu(I_\varepsilon)>0\), where
\[
    I_1=[\ell,r],\quad I_{-1}=[-r,-\ell].
\]
For any $t\ge0$, set
\[
    Z_t
    :=
    \int_0^t\int_{I_\varepsilon}z\,\mathcal N(\dd s,\dd z)
\]
and
\[
    Y_t
    :=
    \int_0^t\int_{\{|z|\le1\}\setminus I_\varepsilon}
        z\,\widetilde{\mathcal N}(\dd s,\dd z)
    +
    \int_0^t\int_{\{|z|>1\}\setminus I_\varepsilon}
        z\,\mathcal N(\dd s,\dd z)
    -
    t\int_{\{|z|\le1\}\cap I_\varepsilon}z\,\nu(\dd z).
\]
Write \(Y:=(Y_t)_{t\ge0}\) and \(Z:=(Z_t)_{t\ge0}\).  Then
\(L_t=Y_t+Z_t\) for every \(t\ge0\), and \(Y\) and \(Z\) are independent.
Since \(Y\) has
c\`adl\`ag paths,
choose \(B<\infty\) such that
\[
    \Pp\left(\sup_{0\le t\le1}|Y_t|\le B\right)>0 .
\]
Let \(M:=\sup_{y\in[R,n]}|b(y)|<\infty\), and choose \(m\) so large that
\begin{equation}\label{eq:b-bounded-exit-m-choice}
    m\ell>n-R+M+B .
\end{equation}
Let \(S_m\) be the time of the \(m\)-th jump of \(Z\).  The event
\[
    G:=\left\{S_m\le1,\ \sup_{0\le t\le1}|Y_t|\le B\right\}
\]
has positive probability.  On \(G\), if
\(S_m<\tau_{R,-}^x\wedge\tau_{n,+}^x\), then \(X_u^x\in(R,n)\) for
\(0\le u<S_m\).  Using \eqref{eq:b-finite-line-equation}
and \eqref{eq:b-bounded-exit-m-choice},
if
\(\varepsilon=1\), then
\[
    X_{S_m}^x
    =
    x-\int_0^{S_m}b(X_u^x)\,\dd u+Y_{S_m}+Z_{S_m}
    \ge R-MS_m-B+m\ell
    \ge R-M-B+m\ell
    >n;
\]
whereas, if \(\varepsilon=-1\), then
\[
    X_{S_m}^x
    =
    x-\int_0^{S_m}b(X_u^x)\,\dd u+Y_{S_m}+Z_{S_m}
    \le n+MS_m+B-m\ell
    \le n+M+B-m\ell
    <R.
\]
The strict inequalities in both displays follow from
\eqref{eq:b-bounded-exit-m-choice}, and both alternatives contradict\(S_m<\tau_{R,-}^x\wedge\tau_{n,+}^x\).  Therefore
\eqref{eq:b-bounded-exit-unit} holds with \(p:=\Pp(G)>0\).
The proof is complete. \end{proof}

For a \(C^2\)-function \(f\), write formally
\[
    \mathcal I f(x)
    =
    \int_{\R\setminus\{0\}}
    \left(
        f(x+z)-f(x)-z f'(x)\mathbf 1_{\{|z|\le1\}}
    \right)\nu(\dd z).
\]
Define
{\small
\[
\begin{aligned}
\mathcal D:=\biggl\{F\in C^2(\R):\quad
&\sup_{x\in K}\int_{\R\setminus\{0\}}
\left|F(x+z)-F(x)-zF'(x)\mathbf 1_{\{|z|\le1\}}\right|
\nu(\dd z)<\infty\\
&\text{for every compact }K\subset\R\biggr\}.
\end{aligned}
\]
}
For \(F\in\mathcal D\), define
\[
    \mathcal L F(x):=-b(x)F'(x)+\mathcal I F(x),
    \quad x\in\R.
\]

\begin{lemma}[The Lyapunov condition for right-tail return]
\label{lem:b-right-lyapunov-return}
Suppose that there exist \(R\in\R\) and
a nonnegative function \(V\in\mathcal D\)
such that
$
    \lim_{x\to+\infty}V(x)=\infty$ and $\mathcal L V(x)\le0$ for all $x>R$.
Then Condition~$\ref{cond:b-right-return}$ holds.
\end{lemma}

\begin{proof}
Fix \(x>R\) and \(n>x\).  Since \(V\in\mathcal D\), It\^o's formula,
localized in bounded intervals and stopped at
\(t\wedge\tau_{R,-}^x\wedge\tau_{n,+}^x\),
shows
that
\(V(X_{t\wedge\tau_{R,-}^x\wedge\tau_{n,+}^x}^x)\) is a nonnegative local
supermartingale.  It is therefore a supermartingale, and
\[
    \E V(X_{t\wedge\tau_{R,-}^x\wedge\tau_{n,+}^x}^x)
    \le V(x),\quad t\ge0 .
\]
On \(\{\tau_{n,+}^x<\tau_{R,-}^x,\ \tau_{n,+}^x\le t\}\), we have
\(X_{t\wedge\tau_{R,-}^x\wedge\tau_{n,+}^x}^x\ge n\).  Hence
\[
    \left(\inf_{y\ge n}V(y)\right)\cdot
    \Pp(\tau_{n,+}^x<\tau_{R,-}^x,\ \tau_{n,+}^x\le t)
    \le V(x).
\]
Letting \(t\to\infty\), we get
\begin{equation}
\label{eq:b-up-before-return}
    \Pp(\tau_{n,+}^x<\tau_{R,-}^x)
    \le\frac{V(x)}{\inf_{y\ge n}V(y)}
    \xrightarrow[n\to\infty]{}0 .
\end{equation}

By Lemma~\ref{lem:b-no-bounded-trapping},
\(\tau_{R,-}^x\wedge\tau_{n,+}^x<\infty\) almost surely.  Hence, on
\(\{\tau_{R,-}^x=\infty\}\), one necessarily has
\(\tau_{n,+}^x<\tau_{R,-}^x\).  Together with
\eqref{eq:b-up-before-return}, this gives
\[
    \Pp(\tau_{R,-}^x=\infty)
    \le
    \Pp(\tau_{n,+}^x<\tau_{R,-}^x)
    \xrightarrow[n\to\infty]{}0 .
\]
Thus \(\Pp(\tau_{R,-}^x<\infty)=1\) for every \(x>R\), which is exactly
Condition~\ref{cond:b-right-return}.
\end{proof}

\subsection{Conditions~$\ref{cond:b-as-coeff}(ii)$ and $\ref{cond:b-lyapunov}$}
\begin{lemma}\label{lem:b-smooth-tail-functions}
Assume that Conditions~$\ref{cond:b-as-coeff}(ii)$ and $\ref{cond:b-lyapunov}$ hold.  Then
there exist nonnegative
smooth functions \(q_-,q_+:\R\to[0,\infty)\), with
bounded derivatives, such that
\[
    q_-(x)\ge\frac1{g_-(x)},\,\,x\le a_-;
    \quad
    q_+(x)\ge\frac1{g_+(x)},\,\, x\ge a_+,
\]
\[
    \int_{-\infty}^{a_-}q_-(x)\,\dd x<\infty,
    \quad
    \int_{a_+}^\infty q_+(x)\,\dd x<\infty,
\]
and
\[
    q_-(x)\to0,\,\, x\to-\infty;
    \quad
    q_+(x)\to0,\,\,x\to+\infty.
\]
\end{lemma}

\begin{proof}
Choose a nonnegative
\(\rho\in C_c^\infty((0,1))\) such that
\(\int_0^1\rho(s)\,\dd s=1\), and define
\[
    q_+(x)
    :=
    \int_0^1
    \frac{\rho(s)}
         {g_+\bigl(\max\{x-s,a_+\}\bigr)}
    \,\dd s,
    \quad x\in\R ,
\]
and
\[
    q_-(x)
    :=
    \int_0^1
    \frac{\rho(s)}
         {g_-\bigl(\min\{x+s,a_-\}\bigr)}
    \,\dd s,
    \quad x\in\R .
\]
Extend \(\rho\) by zero outside \((0,1)\).  After a change of variables,
\[
    q_+(x)
    =
    \int_{\R}
    \frac{\rho(x-u)}
         {g_+\bigl(\max\{u,a_+\}\bigr)}
    \,\dd u,
    \quad
    q_-(x)
    =
    \int_{\R}
    \frac{\rho(u-x)}
         {g_-\bigl(\min\{u,a_-\}\bigr)}
    \,\dd u .
\]
Hence, for every \(k\ge0\),
\[
    q_+^{(k)}(x)
    =
    \int_{\R}
    \frac{\rho^{(k)}(x-u)}
         {g_+\bigl(\max\{u,a_+\}\bigr)}
    \,\dd u,
    \quad
    q_-^{(k)}(x)
    =
    (-1)^k
    \int_{\R}
    \frac{\rho^{(k)}(u-x)}
         {g_-\bigl(\min\{u,a_-\}\bigr)}
    \,\dd u .
\]
Since \(g_+\) and \(g_-\) are positive and have the stated monotonicity,
\[
    \left\|q_+^{(k)}\right\|_\infty
    \le
    \frac{\left\|\rho^{(k)}\right\|_{L^1}}{g_+(a_+)},
    \quad
    \left\|q_-^{(k)}\right\|_\infty
    \le
    \frac{\left\|\rho^{(k)}\right\|_{L^1}}{g_-(a_-)} .
\]
Thus \(q_-\) and \(q_+\) are smooth with bounded derivatives.  If
\(x\ge a_+\), then
\(\max\{x-s,a_+\}\le x\); since \(g_+\) is nondecreasing,
\[
    q_+(x)
    =
    \int_0^1
    \frac{\rho(s)}
         {g_+\bigl(\max\{x-s,a_+\}\bigr)}
    \,\dd s
    \ge
    \int_0^1\frac{\rho(s)}{g_+(x)}\,\dd s
    =
    \frac1{g_+(x)} .
\]
Similarly, if \(x\le a_-\), then \(\min\{x+s,a_-\}\ge x\); since \(g_-\)
is nonincreasing,
\[
    q_-(x)
    =
    \int_0^1
    \frac{\rho(s)}
         {g_-\bigl(\min\{x+s,a_-\}\bigr)}
    \,\dd s
    \ge
    \int_0^1\frac{\rho(s)}{g_-(x)}\,\dd s
    =
    \frac1{g_-(x)} .
\]
Moreover,
Fubini's theorem implies
\[
    \int_{a_+}^\infty q_+(x)\,\dd x =\int_{a_+}^{a_++1} q_+(x)\,\dd x+ \int_{a_++1}^\infty q_+(x)\,\dd x
    \le
    \frac1{g_+(a_+)}
    +
    \int_{a_+}^\infty\frac{\dd u}{g_+(u)}
    <\infty ,
\]
and
\[
    \int_{-\infty}^{a_-}q_-(x)\,\dd x
    \le
    \frac1{g_-(a_-)}
    +
    \int_{-\infty}^{a_-}\frac{\dd u}{g_-(u)}
    <\infty .
\]
The functions \(q_-\) and \(q_+\) are respectively nondecreasing and
nonincreasing.  Their integrability therefore gives
\[
    q_-(x)\to0\quad x\to-\infty;
    \quad
    q_+(x)\to0\quad x\to+\infty.
\] The proof is complete.
\end{proof}

\begin{lemma}[Construction of \(\Phi\)]
\label{lem:b-tail-scale}
Assume that Conditions~$\ref{cond:b-as-coeff}(ii)$ and $\ref{cond:b-lyapunov}$ hold.  Then
there exists a nonnegative
function \(\Phi\in C_b^2(\R)\cap\mathcal D\) such
that \(\Phi'\) is bounded and Lipschitz,
\[
    \Phi'(x)\to0,\quad |x|\to\infty,
\]
and there exists \(R>0\) such that
\[
    b(x)\Phi'(x)\ge1,\quad |x|\ge R .
\]
Moreover,
\[
    \mathcal I\Phi(x)\longrightarrow0,
    \quad |x|\to\infty .
\]
\end{lemma}

\begin{proof}
Let \(q_-\) and \(q_+\) be given by
Lemma~\ref{lem:b-smooth-tail-functions}.  Choose \(R>0\) so large that
\(-R\le a_-\) and \(R\ge a_+\), and choose
\(\chi\in C^\infty(\R;[0,1])\) such that
\[
    \chi(x)=0,\quad x\le-R;
    \quad
    \chi(x)=1,\quad x\ge R.
\]
Define
\[
    p(x):=(1-\chi(x))q_-(x)+\chi(x)q_+(x),
    \quad
    \Phi(x):=\int_{-\infty}^x p(u)\,\dd u .
\]
Then \(p\) is nonnegative, smooth, bounded, Lipschitz, and integrable on
\(\R\).  Hence \(\Phi\in C_b^2(\R)\), \(\Phi'=p\), and
\(\Phi'(x)\to0\) as \(|x|\to\infty\).  For \(x\le-R\),
\[
    b(x)\Phi'(x)
    =
    b(x)q_-(x)
    \ge
    g_-(x)\frac1{g_-(x)}
    =1,
\]
and the same argument with \(g_+\) gives the inequality for \(x\ge R\).
Moreover,
\[
    \lim_{x\to-\infty}\Phi(x)=0,
    \quad
    \lim_{x\to+\infty}\Phi(x)
    =\int_{\R}p(u)\,\dd u<\infty .
\]
Together with \(\Phi'(x)\to0\) as \(|x|\to\infty\), this gives, for each
fixed \(z\),
\[
    \Phi(x+z)-\Phi(x)-z\Phi'(x)\mathbf 1_{\{|z|\le1\}}
    \longrightarrow0,
    \quad  |x|\to\infty .
\]
Let \(C_\Phi:=\operatorname{Lip}(\Phi')\).  If \(|z|\le1\), then
\[
\begin{aligned}
\left|
    \Phi(x+z)-\Phi(x)-z\Phi'(x)
\right| =
\left|
    z\int_0^1\bigl(\Phi'(x+\theta z)-\Phi'(x)\bigr)\,\dd\theta
\right|
\le
\frac{C_\Phi}{2}z^2 .
\end{aligned}
\]
If \(|z|>1\), then
\[
    |\Phi(x+z)-\Phi(x)|\le 2\|\Phi\|_\infty .
\]
The dominating function
\[
    \frac{C_\Phi}{2}z^2\mathbf 1_{\{|z|\le1\}}
    +2\|\Phi\|_\infty\mathbf 1_{\{|z|>1\}}
\]
is \(\nu\)-integrable.  The dominated convergence theorem gives
\(\mathcal I\Phi(x)\to0\)
as $|x|\rightarrow\infty$.
In particular, \(\Phi\in\mathcal D\).
\end{proof}

\begin{proposition}[Uniform Lyapunov entrance]\label{prop:b-lyapunov-input}
Assume that Condition~$\ref{cond:b-as-coeff}(ii)$ and
Condition~$\ref{cond:b-lyapunov}$ hold.  Then there
exist a compact interval \(K_0\subset\R\) and a constant \(M<\infty\) such that
\[
    \E[T_-^x\wedge\tau_{K_0}^x]\le M,
    \quad x\in\R,
\]
where \(\tau_{K_0}^x:=\inf\{t\ge0:X_t^x\in K_0\}\).
\end{proposition}

\begin{proof}
Let \(\Phi\) be the function given by Lemma~\ref{lem:b-tail-scale}.
Choose
\(R\) so large that \(b(x)\Phi'(x)\ge1\) for \(|x|\ge R\) and
\[
    |\mathcal I\Phi(x)|\le\frac12,\quad |x|\ge R,
\]
and set \(K_0:=[-R,R]\).    Hence, for \(x\notin K_0\),
\begin{equation}
\label{eq:b-LPhi-negative}
\begin{aligned}
    \mathcal L \Phi(x)
    &=
    -b(x)\Phi'(x)
    +\int_{\R\setminus\{0\}}
    \left(
        \Phi(x+z)-\Phi(x)-z\Phi'(x)\mathbf 1_{\{|z|\le1\}}
    \right)\nu(\dd z)
    \\
    &\le -1+\mathcal I\Phi(x)
    \le -\frac12.
\end{aligned}
\end{equation}

If \(x\in K_0\), the desired entrance-time bound is trivial.  Below, we consider the case that $x\notin K_0$. Take
\(n>|x|\), \(m>|x|\), and set
\[
    \sigma_{t,n,m}^x
    :=t\wedge\tau_{K_0}^x\wedge\tau_{-n,-}^x\wedge\tau_{m,+}^x.
\]
Dynkin's formula gives
\[
    \E\Phi(X_{\sigma_{t,n,m}^x}^x)
    =
    \Phi(x)+
    \E\int_0^{\sigma_{t,n,m}^x}\mathcal L\Phi(X_s^x)\,\dd s .
\]
On \(\{s<\sigma_{t,n,m}^x\}\), the path has not entered \(K_0\).  Therefore
\eqref{eq:b-LPhi-negative} gives
\(\mathcal L\Phi(X_s^x)\le-1/2\), and hence
\[
    0\le
    \E\Phi(X_{\sigma_{t,n,m}^x}^x)
    \le
    \Phi(x)-\frac12\E\sigma_{t,n,m}^x .
\]
Thus \(\E\sigma_{t,n,m}^x\le2\|\Phi\|_\infty\).  By
Condition~\ref{cond:b-lyapunov} and
Lemma~\ref{lem:b-no-right-explosion-from-return},
\(\Pp(T_+^x=\infty)=1\).  Therefore, letting first \(m\to\infty\), next
\(t\to\infty\), and finally
\(n\to\infty\), by the monotone convergence theorem, gives
\[
    \E[T_-^x\wedge\tau_{K_0}^x]
    \le 2\|\Phi\|_\infty
    =:M,
    \quad x\notin K_0 .
\] The proof is complete.
\end{proof}

\begin{proposition}[A sufficient condition for right-tail return]
\label{cor:b-right-return-from-osgood}
Assume that Condition~$\ref{cond:b-as-coeff}(ii)$ and
Condition~$\ref{cond:b-lyapunov}$ hold.  Then
Condition~$\ref{cond:b-right-return}$ holds.
Moreover, there exists \(R\in\R\) such that
\(\sup_{x\in \R} \E[\tau_{R,-}^x]<\infty\).
\end{proposition}

\begin{proof}
Let \(K_0=[-R,R]\) be the interval obtained in
Proposition~\ref{prop:b-lyapunov-input}.
Then for 
\(x>R\),
\[
    \E[T_-^x\wedge\tau_{K_0}^x]<\infty,
\]
so \(T_-^x\wedge\tau_{K_0}^x<\infty\) almost surely.  On
\(\{\tau_{K_0}^x<\infty\}\), we have \(\tau_{R,-}^x\le\tau_{K_0}^x<\infty\).  On
\(\{T_-^x<\infty\}\), choose \(n\) so large that \(-n\le R\).  Since
\(T_-^x=\lim_n\tau_{-n,-}^x\), the time \(\tau_{-n,-}^x\) is finite before left
explosion, and therefore \(\tau_{R,-}^x\le\tau_{-n,-}^x<\infty\).  Hence
\[
    \Pp(\tau_{R,-}^x<\infty)=1,\quad x>R,
\]
which is Condition~\ref{cond:b-right-return}.
The argument above and Proposition \ref{prop:b-lyapunov-input} also yield
\[
\sup_{x\in \R} \E[\tau_{R,-}^x] \le \sup_{x\in \R} \E[T_-^x\wedge\tau_{K_0}^x] <\infty.
\]
The proof is complete. 
\end{proof}

\section{Explosion Below a Finite Level}\label{sec:b-compact-input}

The main purpose of this section is to give a uniform positive probability of explosion after the
process enters a right-bounded region.  Condition~\ref{cond:b-as-coeff}(i)
gives two ways to reach the far-left Osgood region: either the effective
downward coefficient \(b+\kappa_+\) is positive everywhere, or finitely many
negative jumps push the path down.  The following comparison then forces
explosion.

\begin{lemma}[Localized Osgood comparison]\label{lem:b-osgood-comparison}
Let \(a_*\ge0\), and let \(h:[a_*,\infty)\to(0,\infty)\) be
nondecreasing such that \(\int_{a_*}^\infty h(u)^{-1}\,\dd u<\infty\).
Let \(\xi=(\xi_t)_{0\le t<T}\) be a real-valued c\`adl\`ag path with
explosion time \(T\in(0,\infty]\).  Let \(s<T\), \(a\ge a_*\), and
\(\rho>\int_a^\infty h(u)^{-1}\,\dd u\).
Suppose that, for
\(s\le t<T\wedge(s+\rho)\), \(-\xi_t\in[a_*,\infty)\)
and
\[
    -\xi_t\ge a+\int_s^t h(-\xi_u)\,\dd u .
\]
Then
\[
    T\le s+\int_a^\infty\frac{\dd u}{h(u)}.
\]
\end{lemma}

\begin{proof}
For \(s\le t<T\wedge(s+\rho)\), set
\[
    I_t:=a+\int_s^t h(-\xi_u)\,\dd u .
\]
Then \(I_t\le-\xi_t\) with \(I_s=a\), and \(t\mapsto I_t\) is absolutely continuous.  Since
\(h\) is nondecreasing,
\[
    I'_t=h(-\xi_t)\ge h(I_t)
    \quad \text{for a.e. }t\in(s,T\wedge(s+\rho)).
\]
Define
\[
    \Theta(r):=\int_a^r\frac{\dd u}{h(u)},
    \quad r\ge a .
\]
Then \((\Theta(I_t))'\ge1\) for a.e. \(t\), hence
\begin{equation}\label{eq:b-theta-lower-bound}
    \Theta(I_t)\ge t-s,
    \quad s\le t<T\wedge(s+\rho).
\end{equation}
If
\[
    T>s+\int_a^\infty\frac{\dd u}{h(u)},
\]
then
\[
    t_0:=s+\int_a^\infty\frac{\dd u}{h(u)}
\]
is strictly smaller than \(T\wedge(s+\rho)\).  Applying
\eqref{eq:b-theta-lower-bound} at
\(t=t_0\) gives
\[
    \Theta(I_{t_0})
    \ge
    t_0-s
    =
    \int_a^\infty\frac{\dd u}{h(u)}.
\]
On the other hand, if  \(t_0<T\), so \(I_{t_0}<\infty\)
and
\[
    \Theta(I_{t_0})
    =
    \int_a^{I_{t_0}}\frac{\dd u}{h(u)}
    <
    \int_a^\infty\frac{\dd u}{h(u)},
\]
which is a contradiction.  This proves the claim.
\end{proof}

\begin{lemma}\label{lem:b-left-tail-entrance}
Assume that Condition~$\ref{cond:b-as-coeff}(ii)$ holds.  Let \(a>0\) satisfy
\[
    -a\le a_-,
    \quad
    \int_a^\infty\frac{\dd u}{g_-(-u)}<1.
\]
Fix \(x\in\R\), let \(S\) be a nonnegative
stopping time
with respect to the filtration \((\mathcal F_t)_{t\ge0}\),
and let \(c>0\).
Suppose that, on an event \(E\),
\begin{equation}\label{eq:b-left-tail-entrance-assumptions}
    S<T_-^x\wedge T_+^x,
    \quad
    X_S^x\le-a-c-1,
    \quad
    \sup_{0\le t\le1}(L_{S+t}-L_S)\le c .
\end{equation}
Then, on \(E\),
\[
    T_-^x
    \le
    S+\int_a^\infty\frac{\dd u}{g_-(-u)}
    <S+1 .
\]
\end{lemma}

\begin{proof}
We work on the event \(E\) throughout the proof.  Set
\[
    \eta:=\inf\{t\ge S:X_t^x\ge-a\}.
\]
If \(\eta<T_-^x\wedge T_+^x\wedge(S+1)\), then
\(X_u^x<-a\le a_-\) for \(S\le u<\eta\), and hence \(b(X_u^x)>0\).
By \eqref{eq:b-left-tail-entrance-assumptions},
\[
    X_\eta^x
    =
    X_S^x-\int_S^\eta b(X_u^x)\,\dd u+L_\eta-L_S
    \le-a-c-1+c<-a,
\]
contrary to the definition of \(\eta\).  Thus, $\eta\ge T_-^x\wedge T_+^x\wedge(S+1)$; that is,
\[
    X_t^x<-a,
    \quad
    S\le t<T_-^x\wedge T_+^x\wedge(S+1).
\]
If \(T_+^x<T_-^x\wedge(S+1)\), the last display gives
\(X_t^x<-a\) for \(S\le t<T_+^x\), contradicting
\(X_t^x\to+\infty\) as \(t\uparrow T_+^x\).  Therefore, on \(E\),
\[
    T_+^x\ge T_-^x\wedge(S+1).
\]
For \(S\le t<T_-^x\wedge(S+1)\),
\eqref{eq:b-finite-line-equation}, \eqref{eq:b-left-tail-entrance-assumptions} and Condition \ref{cond:b-as-coeff}(ii) now give
\[
\begin{aligned}
    -X_t^x
    &=
    -X_S^x+\int_S^t b(X_u^x)\,\dd u-(L_t-L_S)\\
    &\ge
    a+\int_S^t g_-(X_u^x)\,\dd u \\ 
    &=
    a+\int_S^t g_-(-(-X_u^x))\,\dd u . 
\end{aligned}
\]
Lemma~\ref{lem:b-osgood-comparison}, applied with
\(T=T_-^x\wedge T_+^x\), \(s=S\), \(\rho=1\), and
\(h(y)=g_-(-y)\), yields
\[
    T_-^x\wedge T_+^x
    \le
    S+\int_a^\infty\frac{\dd u}{g_-(-u)}
    <S+1 .
\]
Together with \(T_+^x\ge T_-^x\wedge(S+1)\), this implies
\(T_-^x=T_-^x\wedge T_+^x\), and hence proves the claim.
\end{proof}

\begin{lemma}[A small-subordinator event]\label{lem:b-small-subordinator-event}
Let
\(J:=(J_t)_{t\ge0}\) 
be a driftless subordinator with L\'evy measure \(\mu\)
such that
\(\int_{(0,1)}z\,\mu(\dd z)<\infty\).  Then, for every \(t>0\) and every
\(C>0\),
\[
    \Pp(J_t\le C)>0 .
\]
\end{lemma}

\begin{proof}
Fix \(t>0\) and \(C>0\), and let \(\mathcal N_J\) be the Poisson random
measure of the jumps of \(J\).  For \(\varepsilon\in(0,1)\), write
\[
    J_t=J_t^1+J_t^2,
    \quad
    J_t^1
    :=
    \int_0^t\int_{(0,\varepsilon]}z\,\mathcal N_J(\dd s,\dd z),
    \quad
    J_t^2
    :=
    \int_0^t\int_{(\varepsilon,\infty)}z\,\mathcal N_J(\dd s,\dd z).
\]
Then, for all $t\ge0$,
\[
    \Pp(J_t^2=0)
    =
    \exp\{-t\mu((\varepsilon,\infty))\}
    >0,
\]
and
\[
    \E J_t^1
    =
    t\int_{(0,\varepsilon]}z\,\mu(\dd z)\longrightarrow0,
    \quad \varepsilon\downarrow0 .
\]
Choose \(\varepsilon\) so small that
\(\E J_t^1<C\).  Markov's inequality gives
\[
    \Pp(J_t^1\le C)
    \ge
    1-\frac{\E J_t^1}{C}
    >0 .
\]
The two parts
$J_t^1$ and $J_t^2$
are independent, hence
\(\Pp(J_t\le C)>0\).
\end{proof}

\begin{lemma}[Downward oscillation from infinite positive small jumps]
\label{lem:b-infinite-positive-variation-lower-event}
Assume
\[
    \int_{(0,1)}z\,\nu(\dd z)=\infty .
\]
Then, for every \(t_0>0\) and every \(B,C>0\),
\[
    \Pp\left(
        \sup_{0\le t\le t_0}L_t\le B,\quad
        \inf_{0\le t\le t_0}L_t\le -C
    \right)>0 .
\]
\end{lemma}

\begin{proof}
For \(\varepsilon\in(0,1)\), decompose
\[
    L_t
    =
    M_t^\varepsilon+J_t^{\varepsilon,+}-c_\varepsilon t+R_t,\quad t\ge0,
\]
where
\[
\begin{aligned}
    M_t^\varepsilon
    &:=
    \int_0^t\int_{(0,\varepsilon]}z\,
    \widetilde{\mathcal N}(\dd s,\dd z),\\
    J_t^{\varepsilon,+}
    &:=
    \int_0^t\int_{(\varepsilon,\infty)}z\,
    \mathcal N(\dd s,\dd z),\\
    c_\varepsilon
    &:=
    \int_{(\varepsilon,1]}z\,\nu(\dd z),\\
    R_t
    &:=
    \int_0^t\int_{[-1,0)}z\,
    \widetilde{\mathcal N}(\dd s,\dd z)
    +
    \int_0^t\int_{(-\infty,-1)}z\,
    \mathcal N(\dd s,\dd z).
\end{aligned}
\]
Then the three L\'evy processes \((M_t^\varepsilon)_{t\ge0}\), \((J_t^{\varepsilon,+})_{t\ge0}\) and \((R_t)_{t\ge0}\)
above are independent,
\(c_\varepsilon\to\infty\) as \(\varepsilon\downarrow0\), and
\[
    \E\bigl[(M_{t_0}^\varepsilon)^2\bigr]
    =
    t_0\int_{(0,\varepsilon]}z^2\,\nu(\dd z)
    \longrightarrow0.
\]

Since \((R_t)_{t\ge0}\) has c\`adl\`ag paths, choose \(K\ge B/4\) such that
\[
    \Pp\left(\sup_{0\le t\le t_0}R_t\le K\right)>\frac34 .
\]
By the right-continuity at zero of the process $(R_t)_{t\ge0}$, we may then choose
\(\delta\in(0,t_0)\) such that
\[
    \Pp\left(\sup_{0\le t\le\delta}R_t\le\frac B4\right)>\frac34 .
\]
Thus
\[
    \Pp\left(
        \sup_{0\le t\le\delta}R_t\le\frac B4,\quad
        \sup_{0\le t\le t_0}R_t\le K
    \right)>\frac12 .
\]
Choose \(\varepsilon>0\) so small that
\[
    c_\varepsilon\delta\ge C+K+\frac B4
\]
and, by Doob's inequality,
\[
\begin{aligned}
    \Pp\left(
        \sup_{0\le t\le t_0}|M_t^\varepsilon|>\frac B4
    \right)
    &\le
    \frac{64t_0}{B^2}
    \int_{(0,\varepsilon]}z^2\,\nu(\dd z)
    <1 .
\end{aligned}
\]
Finally,
\[
    \Pp(J_{t_0}^{\varepsilon,+}=0)
    =
    \exp\{-t_0\nu((\varepsilon,\infty))\}>0 .
\]
Define
\[
\begin{aligned}
    E_\varepsilon
    &:=
    \left\{
        \sup_{0\le t\le\delta}R_t\le\frac B4,\quad
        \sup_{0\le t\le t_0}R_t\le K
    \right\}\cap
    \left\{
        \sup_{0\le t\le t_0}|M_t^\varepsilon|\le\frac B4
    \right\}
    \cap
    \{J_{t_0}^{\varepsilon,+}=0\}.
\end{aligned}
\]
Recall that
\((M_t^\varepsilon)_{t\ge0}\), \((J_t^{\varepsilon,+})_{t\ge0}\) and \((R_t)_{t\ge0}\) are independent, and so
\(\Pp(E_\varepsilon)>0\).  Since \((J_t^{\varepsilon,+})_{t\ge0}\) 
is nondecreasing, \(J_{t_0}^{\varepsilon,+}=0\) implies
\(J_t^{\varepsilon,+}=0\) for \(0\le t\le t_0\).  Hence, on
\(E_\varepsilon\),
\[
\begin{aligned}
    L_t
    &=
    M_t^\varepsilon-c_\varepsilon t+R_t
    \le
    |M_t^\varepsilon|+R_t
    \le
    \frac B4+\frac B4
    =
    \frac B2<B,
    &&0\le t\le\delta,\\
    L_t
    &=
    M_t^\varepsilon-c_\varepsilon t+R_t
    \le
    |M_t^\varepsilon|+R_t-c_\varepsilon t
    \le
    \frac B4+K-c_\varepsilon\delta
    \le-C,
    &&\delta\le t\le t_0.
\end{aligned}
\]
Hence \(\sup_{0\le t\le t_0}L_t\le B\) and
\(\inf_{0\le t\le t_0}L_t\le-C\) on an event of positive probability.
\end{proof}

\begin{proposition}[Uniform explosion below a level]\label{prop:b-halfline-explosion}
Assume that Condition~\ref{cond:b-as-coeff} holds.  Then for every \(A\in\R\)
there exist \(t_*>0\) and \(p>0\) such that
\[
    \inf_{x\le A}\Pp(T_-^x\le t_*)\ge p .
\]
\end{proposition}

\begin{proof}
By Condition~\ref{cond:b-as-coeff}(ii), \(g_-(-y)\) is positive and
nondecreasing for \(y\ge-a_-\).  On the same interval,
\(b(-y)\ge g_-(-y)\), and
\[
    \int_{-a_-}^\infty\frac{\dd u}{g_-(-u)}<\infty .
\]
Fix \(A\in\R\).  Choose \(a\ge\max\{1,1-a_-\}\) so large that
\[
    \int_a^\infty\frac{\dd u}{g_-(-u)}<1 .
\]

Next, we prove the desired assertion according to three cases.

\smallskip

\noindent{\it Case $(i)$: \(\nu((-\infty,0))=0\), \(\kappa_+<\infty\), and
\(b(x)+\kappa_+>0\) for every \(x\in\R\).}

\smallskip

In this case, we can write
\[
    L_t=J_t-\kappa_+t,
    \quad \text{where}\quad
    J_t
    :=
    \int_0^t\int_{(0,\infty)}z\,\mathcal N(\dd s,\dd z).
\]
Choose \(B>\max\{0,a_--A-1\}\).  Since \(g_-\) is nonincreasing,
\[
    b(y)+\kappa_+\ge g_-(y)\ge g_-(a_-)>0,
    \quad y\le a_- .
\]
On \([a_-,A+B+1]\), the function \(b+\kappa_+\) is continuous and
strictly positive.  Hence
\begin{equation}
\label{eq:b-beta-bound}
    \beta:=\inf_{y\le A+B+1}\bigl(b(y)+\kappa_+\bigr)>0 .
\end{equation}
Choose \(T_0>0\) so large that
\begin{equation}
\label{eq:b-T0-choice}
    A+B-\beta T_0\le-B-a-1,
\end{equation}
and set
\[
    E:=\{J_{T_0+1}\le B\}.
\]
Lemma~\ref{lem:b-small-subordinator-event} gives
\[
    q:=\Pp(E)>0 .
\]
Since \(J\) is nondecreasing
and nonnegative,
on \(E\),
\begin{equation}
\label{eq:b-J-increment-bound}
\begin{aligned}
    J_t&\le B,
    &&0\le t\le T_0+1,\\
    J_t-J_{T_0}&\le B,
    &&T_0\le t\le T_0+1 .
\end{aligned}
\end{equation}
Fix \(x\le A\).
On \(E\), if \(T_-^x\le T_0\), there is nothing
to prove.  Suppose \(T_-^x>T_0\).  Let
\[
    \sigma:=\inf\{ t\in [0,T_0]:   
    t<T_-^x\wedge T_+^x,
        \ X_t^x\ge A+B+1\}.
\]
If \(\sigma\le T_0\), then \(X_u^x<A+B+1\) for \(0\le u<\sigma\), so
\eqref{eq:b-beta-bound} gives
\(b(X_u^x)+\kappa_+\ge\beta\).
By
\eqref{eq:b-beta-bound} and
\eqref{eq:b-J-increment-bound},
\[
    A+B+1
    \le X_\sigma^x
    =
    x-\int_0^\sigma\bigl(b(X_u^x)+\kappa_+\bigr)\,\dd u+J_\sigma
    \le A-\beta\sigma+B
    \le A+B,
\]
which is a contradiction.  Thus \(X_t^x<A+B+1\) for
\(0\le t\le T_0\) with \(t<T_+^x\).  This excludes
\(T_+^x\le T_0\).  Since \(T_-^x>T_0\), we have
\begin{equation}
\label{eq:b-pre-T0-upper-bound}
    X_t^x<A+B+1,
    \quad 0\le t\le T_0 .
\end{equation}
By \eqref{eq:b-beta-bound} and \eqref{eq:b-pre-T0-upper-bound},
\[
    b(X_u^x)+\kappa_+\ge\beta,
    \quad 0\le u\le T_0 .
\]
Hence
\begin{equation}
\label{eq:b-XT0-bound}
\begin{aligned}
    X_{T_0}^x
    &=
    x-\int_0^{T_0}\bigl(b(X_u^x)+\kappa_+\bigr)\,\dd u+J_{T_0}\\
    &\le A-\beta T_0+B
    \le-B-a-1 .
\end{aligned}
\end{equation}
The last inequality follows from \eqref{eq:b-T0-choice}.
Moreover, \eqref{eq:b-J-increment-bound} and \(\kappa_+\ge0\) give
\[
    \sup_{0\le t\le1}(L_{T_0+t}-L_{T_0})
    \le B .
\]
Thus \eqref{eq:b-XT0-bound} and
Lemma~\ref{lem:b-left-tail-entrance}, applied with \(S=T_0\) and \(c=B\),
give \(T_-^x<T_0+1\).
Consequently,
\[
    \inf_{x\le A}\Pp(T_-^x\le T_0+1)\ge q>0 .\]

\smallskip

\noindent{Case $(ii)$: \(\nu((-\infty,0))>0\).}

\smallskip

Choose \(0<\ell<r<\infty\) such that
\(\nu([-r,-\ell])>0\).  Decompose
\[
    L_t=Y_t+Z_t,\quad
    Z_t:=\sum_{0<s\le t}\Delta L_s\,
    \mathbf 1_{\{\Delta L_s\in[-r,-\ell]\}} ,\quad t\ge0.
\]
Then \(Z:=(Z_t)_{t\ge0}\) is a compound Poisson process, independent of the residual
L\'evy process
\(Y:=(Y_t)_{t\ge 0}\).
If the interval \([-\!r,-\ell]\) intersects the
small-jump region, the corresponding compensation drift
\(-t \int_{[-r,-\ell]\cap [-1,1]}z \,\nu(\dd z)\)
is absorbed into
\(Y\).  Let
\[
    N_t:=\#\{s\in (0,t]  
    :\Delta L_s\in[-r,-\ell]\},
    \quad
    S_m:=\inf\{t>0:N_t= m\}.
\]
Since \(Y\) has c\`adl\`ag paths, we may choose \(C>0\) such that
\[
    \Pp\left(\sup_{0\le t\le2}|Y_t|\le C\right)>0 .
\]
Choose \(M>\max\{A+C+1,a_-+1\}\).  Since \(b(x)\ge g_-(x)>0\) on
\((-\infty,a_-]\) and \(b\) is continuous on \([a_-,M]\),
\[
    b_M:=
    \inf_{y\le M}b(y)>-\infty .
    \]
Choose \(0<\delta\le1\) so small that
\begin{equation}
\label{eq:b-negative-jump-delta}
    A+C+
    |b_M|
    \delta<M .
\end{equation}
Choose \(m\) so large that
\begin{equation}
\label{eq:b-negative-jump-m}
    m\ell\ge A+3C+|b_M|
    +a+1 .
\end{equation}
Then
\begin{equation}
\label{eq:b-negative-jump-event}
    G:=
    \left\{
        S_m\le\delta,\quad \sup_{0\le t\le2}|Y_t|\le C
    \right\}
\end{equation}
has positive probability, by independence of \(Y\) and \(Z\) and by the Poisson
law of
\(N:=(N_t)_{t\ge 0}\).

Fix \(x\le A\) and work on \(G\).  If \(T_-^x\le S_m\), then
\(T_-^x\le1\) and there is nothing to prove.  Hence suppose \(T_-^x>S_m\).  For
\(0\le t\le S_m\) with \(t<T_-^x\wedge T_+^x\), the path cannot
reach \(M\).  Indeed, if \(\eta\le S_m\), \(\eta<T_-^x\wedge T_+^x\),
were the first time with \(X_\eta^x\ge M\), then
\(X_u^x<M\) for \(u<\eta\), so
\(b(X_u^x)\ge
b_M\)
throughout the integral below.  Since
\(Z_\eta\le0\)
by the definition of \(Z\),
\eqref{eq:b-negative-jump-event} and \eqref{eq:b-negative-jump-delta} give
\[
    M\le X_\eta^x
    \le A+C+
    |b_M|
    \eta
    \le A+C+
    |b_M|
    \delta<M,
\]
a contradiction.  This upper bound before explosion excludes
\(T_+^x\le S_m\).  Since \(T_-^x>S_m\), it follows that
\begin{equation}
\label{eq:b-negative-jump-pre-Sm-bound}
    X_t^x<M,
    \quad 0\le t\le S_m .
\end{equation}
Equations~\eqref{eq:b-negative-jump-pre-Sm-bound} and \eqref{eq:b-negative-jump-event} give
\begin{equation}\label{eq:b-negative-jump-XSm-bound}
\begin{aligned}
    X_{S_m}^x
    &=
    x-\int_0^{S_m}b(X_u^x)\,\dd u+Y_{S_m}+Z_{S_m}\\
    &\le
    A+C+
    |b_M|
    -m\ell\\
    &\le
    -a-2C-1,
\end{aligned}
\end{equation}
where the last line follows from \eqref{eq:b-negative-jump-m}.  Moreover,
\(Z\) is nonincreasing and \(S_m\le1\) on \(G\), so
\eqref{eq:b-negative-jump-event} yields
\[
    \sup_{0\le t\le1}(L_{S_m+t}-L_{S_m})
    \le
    \sup_{0\le t\le1}(Y_{S_m+t}-Y_{S_m})
    \le2C .
\]
Thus \eqref{eq:b-negative-jump-XSm-bound} and
Lemma~\ref{lem:b-left-tail-entrance}, applied with \(S=S_m\) and \(c=2C\),
give \(T_-^x<S_m+1\le2\).  Hence
\[
    \inf_{x\le A}\Pp(T_-^x\le2)\ge\Pp(G)>0 .
\]

\smallskip

\noindent{\it Case $(iii)$: \(\nu((-\infty,0))=0\) and \(\kappa_+=\infty\).}

\smallskip

Choose \(B>0\) such that
\[
    q:=\Pp\left(\sup_{0\le t\le1}L_t\le B\right)>0 .
\]
Choose \(M>\max\{A+B+1,a_-+1\}\).  Since \(b(x)\ge g_-(x)>0\) on
\((-\infty,a_-]\) and \(b\) is continuous on \([a_-,M]\),
\begin{equation}
\label{eq:b-infinite-variation-bM}
    b_M:=\inf_{y\le M}b(y)>-\infty .
\end{equation}
Choose \(0<\delta\le1\) so small that
\begin{equation}
\label{eq:b-infinite-variation-delta}
    A+B+|b_M|\delta<M ,
\end{equation}
and choose \(D\) so large that
\begin{equation}
\label{eq:b-infinite-variation-D}
    D\ge A+B+
    |b_M|+a+1 .
\end{equation}
Set
\[
    S:=\inf\{t\ge0:L_t\le-D\},
\]
and define
\begin{equation}
\label{eq:b-infinite-variation-E0}
    E_0:=
    \left\{
        S\le\delta,\quad
        \sup_{0\le t\le S}L_t\le B
    \right\}.
\end{equation}
Then \(\Pp(E_0)>0\), because
\[
E_0
\supset
\left\{ \inf_{0 \le t \le \delta}L_t \le -D, \quad \sup_{0 \le t \le \delta} L_t \le B \right\}
\]
and the event at the right hand side above has positive probability by Lemma~\ref{lem:b-infinite-positive-variation-lower-event}.
Moreover, \(E_0\in\mathcal F_S\).  By the strong Markov property of
\(L\) at \(S\), the event
\begin{equation}
\label{eq:b-infinite-variation-E1}
    E_1:=\left\{
        \sup_{0\le t\le1}(L_{S+t}-L_S)\le B
    \right\}
\end{equation}
has conditional probability \(q\) given \(\mathcal F_S\) on \(\{S<\infty\}\).
Hence \(E:=E_0\cap E_1\) has positive probability.

Fix \(x\le A\) and work on \(E\).  If \(T_-^x\le S\), then
\(T_-^x\le1\) and there is nothing to prove.  Suppose \(T_-^x>S\).  For
\(0\le t\le S\) with \(t<T_-^x\wedge T_+^x\), the path cannot
reach \(M\).  Indeed, if \(\eta\le S\), \(\eta<T_-^x\wedge T_+^x\),
were the first time with \(X_\eta^x\ge M\), then
\(X_u^x<M\) for \(u<\eta\), so \eqref{eq:b-infinite-variation-bM} gives
\(b(X_u^x)\ge b_M\) throughout the integral below.  Equations
\eqref{eq:b-infinite-variation-E0} and
\eqref{eq:b-infinite-variation-delta} then give
\[
    M\le X_\eta^x
    \le A+B+|b_M|\eta
    \le A+B+|b_M|\delta<M,
\]
a contradiction.  This upper bound before explosion excludes
\(T_+^x\le S\).  Since \(T_-^x>S\), it follows that
\begin{equation}
\label{eq:b-infinite-variation-pre-S-bound}
    X_t^x<M,
    \quad 0\le t\le S .
\end{equation}
Using \eqref{eq:b-infinite-variation-pre-S-bound},
\eqref{eq:b-infinite-variation-bM},
\eqref{eq:b-infinite-variation-E0} and
\eqref{eq:b-infinite-variation-D}, we obtain
\begin{equation}
\label{eq:b-infinite-variation-XS-bound}
\begin{aligned}
    X_S^x
    &=
    x-\int_0^S b(X_u^x)\,\dd u+L_S\\
    &\le A+|b_M|\delta-D\\
    &\le -B-a-1 .
\end{aligned}
\end{equation}
Moreover, \eqref{eq:b-infinite-variation-E1} gives
\[
    \sup_{0\le t\le1}(L_{S+t}-L_S)\le B .
\]
Thus \eqref{eq:b-infinite-variation-XS-bound} and
Lemma~\ref{lem:b-left-tail-entrance}, applied with \(c=B\), give
\(T_-^x<S+1\le2\).  Consequently,
\[
    \inf_{x\le A}\Pp(T_-^x\le2)\ge\Pp(E)>0 .
\] The proof is complete.
\end{proof}

\section{The Renewal Step and Proofs of Results}\label{sec:b-renewal}

\subsection{Proof of Theorem~$\ref{thm:b-as-explosion}$}
\begin{proposition}[Almost sure half-line renewal criterion]\label{prop:b-as-renewal}
Assume that Condition~$\ref{cond:b-right-return}$ holds, and let \(R\) be a level for
which it holds.  Suppose moreover that there exist \(t_*>0\) and
\(p\in(0,1]\) such that
\[
    \inf_{x\le R}\Pp(T_-^x\le t_*)\ge p .
\]
Then
\[
    \Pp(T_-^x<\infty)=1,\quad x\in\R .
\]
\end{proposition}

\begin{proof}
Fix \(x\in\R\).  By the first part of
Lemma~\ref{lem:b-no-right-explosion-from-return},
\(\Pp(T_+^x=\infty)=1\).  Hence, on \(\{T_-^x=\infty\}\), the process remains
\(\R\)-valued at every finite time.

\begin{samepage}
If \(x\le R\), set \(\theta_0^x:=0\); if \(x>R\), set
\(\theta_0^x:=\tau_{R,-}^x\).  For \(n\ge0\), define
\[
    \theta_{n+1}^x
    :=
    \inf\{t\ge\theta_n^x+t_*:X_t^x\le R\}
\]
with \(\theta_{n+1}^x:=\infty\) when \(\theta_n^x=\infty\).  Each \(\theta_n^x\) is a
stopping time.
\end{samepage}

We first prove that
\begin{equation}\label{eq:stopping-time-finite}
    \Pp(\theta_n^x<\infty)=1,\quad n\ge0.
\end{equation}
This holds for \(n=0\) by Condition~\ref{cond:b-right-return} and the
definition of \(\theta_0^x\).  For the induction step, suppose that
\(\theta_n^x<\infty\) almost surely.  Then \(\theta_n^x+t_*\) is a finite
stopping time.  By Lemma~\ref{lem:b-no-right-explosion-from-return},
\(T_+^x=\infty\) almost surely, and hence
\[
\begin{aligned}
    \Pp(\theta_{n+1}^x=\infty)
    &=
    \Pp\!\left(
        \theta_{n+1}^x=\infty,\ X_{\theta_n^x+t_*}^x\le R
    \right)+
    \Pp\!\left(
        \theta_{n+1}^x=\infty,\ X_{\theta_n^x+t_*}^x>R
    \right)\\
    &=
    \Pp\!\left(
        \theta_{n+1}^x=\theta_n^x+t_*=\infty,\
        X_{\theta_n^x+t_*}^x\le R
    \right)+
    \E\!\left[
        \mathbf 1_{\{X_{\theta_n^x+t_*}^x>R\}}
        \Pp\!\left(
            \tau_{R,-}^{X_{\theta_n^x+t_*}^x}=\infty
        \right)
    \right]\\
    &=0,
\end{aligned}
\]
where, in the last equality, the first term is zero since \(\theta_n^x<\infty\) almost
surely, and the second is zero by the strong Markov property at
\(\theta_n^x+t_*\) and Condition~\ref{cond:b-right-return}.  Thus
\(\theta_{n+1}^x<\infty\) almost surely, and
\eqref{eq:stopping-time-finite} follows by induction.

Set \(A_n:=\{\theta_n^x<T_-^x\}\).
Since $A_{n+1}\subset A_n\cap \{T_-^x>\theta_n^x+t_*\}$, the
strong Markov property at \(\theta_n^x\) yields
\[
\begin{aligned}
    \Pp(A_{n+1})
    &\le
    \E\!\left[
        \mathbf 1_{A_n}
        \mathbf 1_{\{T_-^x>\theta_n^x+t_*\}}
    \right]                                                   \\
    &=
    \E\!\left[
        \mathbf 1_{A_n}
        \Pp\!\left(
            T_-^x>\theta_n^x+t_*
            \,\middle|\,
            \mathcal F_{\theta_n^x}
        \right)
    \right]                                                   \\
    &=
    \E\!\left[
        \mathbf 1_{A_n}
        \Pp\!\left(T_-^{X_{\theta_n^x}^x}>t_*\right)
    \right]                                                   \\
    &\le
    (1-p)\Pp(A_n).
\end{aligned}
\]
Thus \(\Pp(A_n)\le(1-p)^n\).  By
\eqref{eq:stopping-time-finite},
\[
    \Pp(T_-^x=\infty)
    =\Pp(\theta_n^x<\infty,\ T_-^x=\infty)
    \le\Pp(A_n)\le(1-p)^n.
\]
Letting \(n\to\infty\) completes the proof.
\end{proof}

\begin{proof}[Proof of Theorem~$\ref{thm:b-as-explosion}$]
Let \(R\) be the level in Condition~\ref{cond:b-right-return}.  By
Proposition~\ref{prop:b-halfline-explosion}, applied with \(A=R\), there exist
\(t_*>0\) and \(p>0\) such that
\[
    \inf_{x\le R}\Pp(T_-^x\le t_*)\ge p .
\]
Together with Condition~\ref{cond:b-right-return}, this is exactly the hypothesis
of Proposition~\ref{prop:b-as-renewal}.
\end{proof}

\subsection{Proof of Theorem~$\ref{thm:b-main}$}
\begin{proposition}[Renewal criterion]\label{prop:b-renewal}
For $x\in \R$ and  each compact interval \(K\subset\R\), set
\(\tau_K^x:=\inf\{t\ge0:X_t^x\in K\}\).
Suppose that
\[
    \Pp(T_+^x\ge T_-^x)=1,\quad x\in\R,
\]
and that there exist a compact interval \(K\subset\R\) and constants
\(M<\infty\), \(t_*>0\) and \(p>0\), such that
\begin{equation}\label{eq:b-renewal-assumptions}
    \sup_{x\in\R}\E[T_-^x\wedge\tau_K^x]\le M
    \quad \mbox{and}\quad
    \inf_{x\in K}\Pp(T_-^x\le t_*)\ge p .
\end{equation}
Then
\[
    \sup_{x\in\R}\E T_-^x<\infty .
\]
\end{proposition}

\begin{proof}
For \(N\ge0\), set \(F_N(x):=\E[T_-^x\wedge N]\).  Fix \(N>t_*\).
The strong Markov property at \(\tau_K^x\) gives us that for all $x\in \R$,
\begin{equation}\label{eq:b-global-renewal-bound}
\begin{aligned}
    F_N(x)
    &=
    \E[T_-^x\wedge N]                                      \\
    &=
    \E[(T_-^x\wedge N)\mathbf 1_{\{T_-^x\wedge N\le\tau_K^x\}}]
    +\E[(T_-^x\wedge N)\mathbf 1_{\{\tau_K^x<T_-^x\wedge N\}}] \\
    &=
    \E[(T_-^x\wedge N)\mathbf 1_{\{T_-^x\wedge N\le\tau_K^x\}}] \\
    &\quad+
    \E\!\left[
        \mathbf 1_{\{\tau_K^x<T_-^x\wedge N\}}
        \left(\tau_K^x+
        \bigl((T_-^x-\tau_K^x)\wedge(N-\tau_K^x)\bigr)\right)
    \right]                                      \\
    &=
    \E[T_-^x\wedge\tau_K^x\wedge N]
    +\E\!\left[
        \mathbf 1_{\{\tau_K^x<T_-^x\wedge N\}}
        \bigl((T_-^x-\tau_K^x)\wedge(N-\tau_K^x)\bigr)
    \right]                                      \\
    &\le
    \E[T_-^x\wedge\tau_K^x]
    +\E\!\left[
        \mathbf 1_{\{\tau_K^x<T_-^x\wedge N\}}
         \bigl((T_-^x-\tau_K^x)\wedge N\bigr)
    \right]                                      \\
    &=
    \E[T_-^x\wedge\tau_K^x]
    +\E\!\left[
        \mathbf 1_{\{\tau_K^x<T_-^x\wedge N\}}
        F_N(X_{\tau_K^x}^x)
    \right]                                      \\
    &\le M+\sup_{y\in K}F_N(y).
\end{aligned}
\end{equation}
If \(x\in K\), then the
Markov property at time \(t_*\),
\eqref{eq:b-global-renewal-bound}, and \eqref{eq:b-renewal-assumptions} give
\[
\begin{aligned}
    F_N(x)
    &=
    \E[T_-^x\mathbf 1_{\{T_-^x\le t_*\}}]
    +
    \E\!\left[
        \mathbf 1_{\{T_-^x>t_*\}}
        \left(t_*+F_{N-t_*}(X_{t_*}^x)\right)
    \right]                                      \\
    &\le
    t_*+
    \E\!\left[
        \mathbf 1_{\{T_-^x>t_*\}}F_N(X_{t_*}^x)
    \right]                                      \\
    &\le
    t_*+(1-p)\left(M+\sup_{y\in K}F_N(y)\right).
\end{aligned}
\]
Taking the supremum over \(x\in K\) and rearranging,
\[
    \sup_{x\in K}F_N(x)
    \le
    \frac{t_*+(1-p)M}{p}.
\]
Together with \eqref{eq:b-global-renewal-bound}, this gives
\[
    \sup_{x\in\R}F_N(x)
    \le
    M+\frac{t_*+(1-p)M}{p}
    =
    \frac{M+t_*}{p},
\]
uniformly in \(N\).  Letting \(N\to\infty\) proves the claim.
\end{proof}

\begin{proof}[Proof of Theorem~$\ref{thm:b-main}$]
Since Conditions~\ref{cond:b-as-coeff} and \ref{cond:b-lyapunov} hold by assumption,
Proposition~\ref{cor:b-right-return-from-osgood} implies that Condition~\ref{cond:b-right-return} also holds.
Then Lemma~\ref{lem:b-no-right-explosion-from-return} gives
\(\Pp(T_+^x=\infty)=1\) for every \(x\in\R\).
Proposition~\ref{prop:b-lyapunov-input} gives a compact interval \(K_0\) and
a constant \(M<\infty\) such that
\[
    \E[T_-^x\wedge\tau_{K_0}^x]\le M,
    \quad x\in\R .
\]
Choose \(A_0<\infty\) with \(K_0\subset(-\infty,A_0]\).  By
Proposition~\ref{prop:b-halfline-explosion}, applied with \(A=A_0\), we get
\[
    \inf_{x\in K_0}\Pp(T_-^x\le t_*)\ge p
\]
for some \(t_*>0\) and \(p>0\).  Proposition~\ref{prop:b-renewal} now gives
\(\sup_{x\in\R}\E T_-^x<\infty\).
\end{proof}

\section{Converse Criteria and Riccati Equations}
\label{sec:b-converse}

\subsection{Pathwise Osgood converses}

\begin{lemma}
\label{lem:b-pathwise-converses}
Fix a c\`adl\`ag path $\ell:[0,\infty)\to\R$ with $\ell_0=0$.
For each $x\in\R$, let $X^x:=(X^x_t)_{t\ge0}$ be the solution up to explosion of the following equation
\begin{equation}\label{eq:b-pathwise-forcing}
    X_t^x=x-\int_0^t b(X_s^x)\,\dd s+\ell_t,
    \quad 0\le t<T_-^x\wedge T_+^x,
\end{equation}
where $T_-^x$ and $T_+^x$ are its left and right explosion times, respectively.
\begin{enumerate}
\renewcommand{\labelenumi}{\textup{(\roman{enumi})}}
\item Suppose that $b$ is positive and nonincreasing on $(-\infty,a_-]$.
      If $\int_{-\infty}^{a_-}b(u)^{-1}\,\dd u=\infty$, then
      $T_-^x=\infty$ for every $x\in\R$.
\item Suppose that $b$ is positive and nondecreasing on $[a_+,\infty)$.
      If $\int_{a_+}^{\infty}b(u)^{-1}\,\dd u=\infty$, then
      $\tau_{R,-}^x\to\infty$ as $x\to\infty$ for every $R\in\R$.
\end{enumerate}
\end{lemma}

\begin{proof}
Fix $H>0$, and write $M:=\sup_{0\le t\le H}|\ell_t|$ and
$T^x:=T_-^x\wedge T_+^x$.
Set $Z_t:=X_t^x-\ell_t$ for $0\le t<T^x$.  By
\eqref{eq:b-pathwise-forcing},
\begin{equation}\label{eq:b-pathwise-forcing-1} 
    Z_t=x-\int_0^t b(X_s^x)\,\dd s,\quad 0\le t<T^x.
\end{equation} 
Thus $t\mapsto Z_t$ is locally absolutely continuous, and satisfies
\[
    Z_0=x,\quad Z'_t=-b(Z_t+\ell_t)\quad\text{for a.e. } 0\le t<T^x.
\]

For \textup{(i)}, if $T_+^x<\infty$, then
$T_-^x=\infty$ by the absorbing convention.  We may therefore
assume $T_+^x=\infty$, so $T^x=T_-^x$.  Choose $K$ such that
$K+M\le a_-$, and define
\[
    \Theta_-(z)=
    \begin{cases}
    \displaystyle\int_z^K\frac{\dd u}{b(u-M)},&\quad z<K,\\
    0,&\quad z\ge K.
    \end{cases}
\]
On $\{0\le t<H\wedge T^x:Z_t<K\}$,
\[
    Z_t-M\le Z_t+\ell_t\le K+M\le a_-.
\]
The positivity and monotonicity of $b$ therefore give
$0<b(Z_t+\ell_t)\le b(Z_t-M)$, and hence
\begin{equation}\label{eq:b-left-theta-derivative}
    Z'_t=-b(Z_t+\ell_t)<0,\quad
    \frac{\dd}{\dd t}\Theta_-(Z_t)
    =\frac{b(Z_t+\ell_t)}{b(Z_t-M)}\le1
\end{equation}
for almost every $t$ in this set.  On
$\{0\le t<H\wedge T^x:Z_t\ge K\}$, we have $\Theta_-(Z_t)=0$, and its
derivative vanishes almost everywhere on this set.  The two cases cover
$[0,H\wedge T^x)$, so integrating the
second 
derivative bound
in \eqref{eq:b-left-theta-derivative} 
gives
\begin{equation}\label{eq:b-left-theta-bound}
    \Theta_-(Z_t)\le\Theta_-(x)+t,
    \quad 0\le t<H\wedge T^x.
\end{equation}
Since $K-M\le a_-$, the divergent integral in \textup{(i)} implies
$\Theta_-(z)\to\infty$ as $z\to-\infty$.
Assume, to the contrary, that $T_-^x<H$.
Since $\ell$ is a c\`adl\`ag path, we have $|\ell_{T_-^x-}|<\infty$ and
so $Z_t\to-\infty$ as $t\uparrow T_-^x$.
Letting $t\uparrow T_-^x$ in
\eqref{eq:b-left-theta-bound} gives
$\infty\le\Theta_-(x)+T_-^x<\infty$, a contradiction.
As $H$ was arbitrary, $T_-^x=\infty$.

For \textup{(ii)}, since $b$ is positive on $[a_+,\infty)$,
Lemma~\ref{lem:b-no-right-explosion-from-return} gives $T_+^x=\infty$,
so $T^x=T_-^x$.  Fix $R\in\R$ and choose
$K>\max\{a_++M,R+M\}$.  For $z\ge K$, put
\[
    \Theta_+(z)=\int_K^z\frac{\dd u}{b(u+M)}.
\]
Start at $x>K$ and set
\[
    \sigma_K:=\inf\{t<T^x:Z_t\le K\}.
\]
By \eqref{eq:b-pathwise-forcing-1},  
$t\mapsto Z_t$ is continuous.  Thus
$Z_{\sigma_K}=K$ whenever
$\sigma_K<\infty$.
A left explosion must be preceded by $Z$ reaching $K$,
so $\sigma_K\le T^x$.  For $0\le t<\sigma_K\wedge H$,
$Z_t>K>a_++M$ and $-M\le\ell_t\le M$.  Hence
\[
    a_+<K-M<Z_t-M\le Z_t+\ell_t\le Z_t+M.
\]
Positivity and monotonicity imply
\begin{equation}\label{eq:b-right-theta-derivative}
    Z'_t=-b(Z_t+\ell_t) < 
    0,\quad
    \frac{\dd}{\dd t}\Theta_+(Z_t)
    =-\frac{b(Z_t+\ell_t)}{b(Z_t+M)}\ge-1.
\end{equation}
In particular, for $t\in [0,\sigma_K\wedge H)$, since $X_t^x >a_+$,  \eqref{eq:b-pathwise-forcing-1} yields $K<Z_t\le x$.
Integrating
the second derivative bound in  
\eqref{eq:b-right-theta-derivative} gives
\[
    \Theta_+(Z_t)\ge\Theta_+(x)-t,
    \quad 0\le t<\sigma_K\wedge H.
\]
We show by contradiction that $\sigma_K>H$ for all sufficiently large $x$.
If $\sigma_K\le H$, letting $t\uparrow\sigma_K$
in the preceding bound gives
\[
    0=\Theta_+(K)\ge\Theta_+(x)-\sigma_K
    \ge\Theta_+(x)-H.
\]
Since $\Theta_+(x)\to\infty$
as $x\to\infty$, this inequality cannot hold for
sufficiently large $x$.  Therefore $\sigma_K>H$ for all such $x$.
Hence $T_-^x>H$, and on $[0,H]$ we have
$X_t^x=Z_t+\ell_t\ge Z_t-M>K-M>R$.
It follows that $\tau_{R,-}^x>H$.
As $H>0$ was arbitrary, $\tau_{R,-}^x\to\infty$ as $x\to\infty$,
which proves \textup{(ii)}.
\end{proof}

\subsection{The sharp criteria}

\begin{proof}[Proof of Theorem~$\ref{thm:b-sharp-osgood}$]
\textup{(i)} Suppose that Conditions~\ref{cond:b-right-return}
and~\ref{cond:b-as-coeff}(i) hold, and that $I_-<\infty$.
Taking $g_-=b$ on the left tail verifies Condition~\ref{cond:b-as-coeff}(ii).
Theorem~\ref{thm:b-as-explosion} therefore gives
$\Pp(T_-^x<\infty)=1$ for every $x\in\R$.

Conversely, suppose that $\Pp(T_-^x<\infty)=1$ for every $x\in\R$.
By Remark~\ref{rem:b-assumption-roles}(i) and (ii),
Conditions~\ref{cond:b-right-return} and~\ref{cond:b-as-coeff}(i) hold.
If $I_-=\infty$, Lemma~\ref{lem:b-pathwise-converses}(i) excludes left
explosion.  This contradicts the assumed almost-sure left explosion,
so $I_-<\infty$.

\textup{(ii)} Suppose that Condition~\ref{cond:b-as-coeff}(i) holds and
$I_-+I_+<\infty$.  Taking $g_-=b$ and $g_+=b$ on the respective tails
verifies Conditions~\ref{cond:b-as-coeff}(ii) and~\ref{cond:b-lyapunov}.
Theorem~\ref{thm:b-main} therefore yields
$\sup_{x\in\R}\E T_-^x<\infty$.

Conversely, suppose that $\sup_{x\in\R}\E T_-^x<\infty$.
Then $\Pp(T_-^x<\infty)=1$ for every $x\in\R$, so part~\textup{(i)}
gives Condition~\ref{cond:b-as-coeff}(i) and $I_-<\infty$.
If $I_+=\infty$, fix $R\in\R$ and set $x_n=R+n$.
Lemma~\ref{lem:b-pathwise-converses}(ii), applied to the common c\`adl\`ag
noise path, gives $\tau_{R,-}^{x_n}\to\infty$ almost surely.
Since $T_-^{x_n}\ge\tau_{R,-}^{x_n}$, Fatou's lemma yields
\[
    \infty=\E\liminf_{n\to\infty}\tau_{R,-}^{x_n}
    \le\liminf_{n\to\infty}\E T_-^{x_n}
    \le\sup_{x\in\R}\E T_-^x<\infty.
\]
This contradicts the assumed uniform bound on the mean explosion time.
Hence $I_+<\infty$.
\end{proof}

\subsection{Riccati equations and the nonexplosive barrier}

\begin{proof}[Proof of Corollary~$\ref{cor:b-riccati-classification}$]
Since $b(y)=y^2+a$ is nonnegative for all sufficiently large $y$,
Lemma~\ref{lem:b-no-right-explosion-from-return} gives $T_+^x=\infty$
almost surely for every $x\in\R$.

\textup{(i)} For every $a\in\R$, the drift $b(x)=x^2+a$ is positive and monotone
in the required directions sufficiently far out on the two tails, and
both Osgood integrals are finite.  Condition~\ref{cond:b-as-coeff}(i)
reduces exactly to \eqref{eq:b-riccati-access}.  If this condition holds,
Theorem~\ref{thm:b-sharp-osgood}(ii) gives a uniform bound on the mean explosion time,
and hence almost-sure left explosion from every initial state.

\textup{(ii)} Suppose that \eqref{eq:b-riccati-access} fails.  There are no negative jumps,
$\kappa_+<\infty$, and $a+\kappa_+\le0$.  Let
\[
    r=\sqrt{-(a+\kappa_+)},\quad J_t=L_t+\kappa_+t, 
\]
and $c_r(y)=r^2-y^2$ for $y\in \R$. 
Then $(J_t)_{t\ge0}$ is a nontrivial driftless subordinator and
\begin{equation}\label{eq:b-riccati-subordinator}
    X_t^x=x+\int_0^t\bigl(r^2-(X_s^x)^2\bigr)\,\dd s+J_t
    =x+\int_0^t c_r(X_s^x)\,\dd s+J_t. 
\end{equation}
Since all jumps are nonnegative and the drift
$c_r(y)$ 
is zero at
$y=-r$,  
solutions of \eqref{eq:b-riccati-subordinator} starting at or above $-r$ cannot cross below $-r$.
Consequently,
\begin{equation}\label{eq:b-riccati-no-left-explosion}
    \Pp(T_-^x<\infty)=0,\quad x\ge-r,
\end{equation}
which is the first assertion in \eqref{eq:b-riccati-failure-regimes}.

Fix $x<-r$.  We first prove that $\Pp(T_-^x<\infty)>0$ by comparison
with a deterministic equation.
Choose $\varepsilon>0$ with $x+\varepsilon<-r$.
Let $(Y_t)_{t\in [0, \zeta)}$ be the  solution to the deterministic Riccati equation
\begin{equation}\label{eq:b-riccati-deterministic}
    Y_t=x+\varepsilon+\int_0^t\bigl(r^2-Y_s^2\bigr)\,\dd s,
    \quad 0\le t<\zeta,
\end{equation}
where $\zeta:=\int_{-\infty}^{x+\varepsilon} \frac{\dd u}{u^2-r^2}<\infty$ is the explosion time of the equation \eqref{eq:b-riccati-deterministic}.
Since $x+\varepsilon<-r$, the solution stays below $-r$ and decreases
to $-\infty$.
On the event $\{J_\zeta\le\varepsilon\}$, we claim that
$X_t^x\le Y_t$ for $0\le t<T_-^x\wedge\zeta$.
Suppose, to the contrary, that this fails, and define
\[
    \rho:=\inf\{s\in[0,T_-^x\wedge\zeta):X_s^x\ge Y_s\}.
\]
Then $X_\rho^x\ge Y_\rho$, and
$X_s^x<Y_s<-r$ for $0\le s<\rho$.
Since $J_\rho\le \varepsilon$ on $\{J_\zeta\le\varepsilon\}$, subtracting
\eqref{eq:b-riccati-subordinator} from \eqref{eq:b-riccati-deterministic}
at time $\rho$ gives
\[
    Y_\rho-X_\rho^x=\varepsilon-J_\rho
        +\int_0^\rho\bigl((X_s^x)^2-Y_s^2\bigr)\,\dd s>0.
\]
This contradicts $X_\rho^x\ge Y_\rho$. Then the comparison $X_t^x\le Y_t$  holds for all $0\le t<T_-^x\wedge\zeta$, which yields
\[
    \{J_\zeta\le\varepsilon\}
    \subseteq\{T_-^x\le\zeta\}\subseteq\{T_-^x<\infty\}.
\]
Lemma~\ref{lem:b-small-subordinator-event} now gives
\[
    \Pp(T_-^x<\infty)\ge\Pp(J_\zeta\le\varepsilon)>0.
\]

We next prove that $\Pp(T_-^x<\infty)<1$.
By \eqref{eq:b-riccati-no-left-explosion} and the strong Markov property,
it suffices to show that $X^x$ reaches $[-r,\infty)$ with positive
probability.  Set
\[
    \sigma:=\inf\{t\ge0:X_t^x\ge-r\},\quad
    \tau:=\inf\{t\ge0:X_t^x\le x-1\}.
\]
Put $K:=\max_{u\in[x-1,-r]}(u^2-r^2)>0$ and choose $h>0$ with $Kh<1$.
We claim that $\tau>h$ on $\{\sigma>h\}$.
Suppose, to the contrary, that $\tau\le h$ on this event.
Then $X_s^x\in(x-1,-r)$ for $0\le s<\tau$.
Since all jumps are nonnegative, \eqref{eq:b-riccati-subordinator} gives
\[
\begin{aligned}
    x-1=X_\tau^x
    &=x+\int_0^\tau\bigl(r^2-(X_s^x)^2\bigr)\,\dd s+J_\tau\\
    &\ge x-K\tau+J_\tau>x-1.
\end{aligned}
\]
This contradiction shows that $\tau>h$ on $\{\sigma>h\}$.
Hence, on this event,
\eqref{eq:b-riccati-subordinator} gives
\[
    -r>X_h^x\ge x-Kh+J_h.
\]
Thus $\{\sigma>h\}\subseteq\{J_h<-r-x+Kh\}$, and therefore
\[
    \{J_h>-r-x+Kh\}\subseteq\{\sigma\le h\}.
\]
Since $J$ is nontrivial, jumps above some fixed positive size occur
at a positive Poisson rate.  Thus
$\Pp(J_h>-r-x+Kh)>0$.  Consequently,
\[
    \Pp(T_-^x=\infty)\ge\Pp(\sigma\le h)
    \ge\Pp(J_h>-r-x+Kh)>0.
\]
This completes the proof.
\end{proof}

\section{SPDE Explosion and a Concrete Application}
\label{sec:spde-application}
We first prove Theorem~\ref{cor:spde-explicit-correction},
and 
then construct a local weak
solution up to explosion in
Corollary~\ref{cor:spde-poisson-maximal}.
Throughout the comparison argument, \((u,T)\) denotes a local weak solution
up to explosion as defined above.

Let
\[
    \varphi(x):=\frac{\pi}{2}\sin(\pi x),\quad x\in(0,1).
\]
Then \(\varphi>0\) on \((0,1)\), \(\varphi(0)=\varphi(1)=0\), and
\[
    -\varphi''=\pi^2\varphi,\quad
    \int_0^1\varphi(x)\,\d x=1.
\]
Define
\[
    \Phi(t):=\langle u(t),\varphi\rangle,
    \quad 0\le t<T.
\]

For every Borel set \(B\subset\R\setminus\{0\}\), set
\begin{equation}\label{eq:spde-projected-levy-measure}
    \lambda_{\sigma,\varphi}(B)
    :=
    \int_0^1\int_{\R\setminus\{0\}}
    \mathbf 1_B\bigl(\sigma\varphi(x)y\bigr)
    \,\lambda(\d y)\,\d x
\end{equation}
and
\[
\gamma_{\sigma,\varphi}
    :=
    \int_0^1\int_{\R\setminus\{0\}}
    \sigma\varphi(x)y
    \left(
        \mathbf 1_{\{|\sigma\varphi(x)y|\le1\}}
        -\mathbf 1_{\{|y|\le1\}}
    \right)
    \lambda(\d y)\,\d x.
\]
Let us
verify that \(\lambda_{\sigma,\varphi}\) is a L\'evy measure and that
\(\gamma_{\sigma,\varphi}\) is absolutely finite. Indeed, by \eqref{eq:spde-projected-levy-measure},
\[
\begin{aligned}
\int_{\R\setminus\{0\}}(1\wedge w^2)\,
\lambda_{\sigma,\varphi}(\d w)
&=
\int_0^1\int_{\R\setminus\{0\}}
\bigl(1\wedge (\sigma^2\varphi(x)^2y^2)\bigr)\,\lambda(\d y)\,\d x\\
&\le
\bigl(1\vee\sigma^2\|\varphi\|_\infty^2\bigr)
\int_{\R\setminus\{0\}}(1\wedge y^2)\,\lambda(\d y)<\infty.
\end{aligned}
\]
Thus \(\lambda_{\sigma,\varphi}\) is a L\'evy measure.
Moreover,
\[
\begin{aligned}
&\int_0^1\int_{\R\setminus\{0\}}\sigma\varphi(x)|y|
\left|
\mathbf 1_{\{\sigma\varphi(x)|y|\le1\}}-\mathbf 1_{\{|y|\le1\}}
\right|\lambda(\d y)\,\d x\\
&\quad={}\int_{\{|y|>1\}}|y|\int_0^1
\sigma\varphi(x)\mathbf 1_{\{\sigma\varphi(x)|y|\le1\}}\,\d x\,\lambda(\d y)+
\int_{\{|y|\le1\}}|y|\int_0^1
\sigma\varphi(x)\mathbf 1_{\{\sigma\varphi(x)|y|>1\}}\,\d x\,\lambda(\d y).
\end{aligned}
\]
For the first term, \(\sigma\varphi(x)|y|\le1\) on the indicated set, so
\[
\begin{aligned}
\int_{\{|y|>1\}}|y|\int_0^1
\sigma\varphi(x)
\mathbf 1_{\{\sigma\varphi(x)|y|\le1\}}\,\d x\,\lambda(\d y)
\le\lambda(\{|y|>1\})<\infty.
\end{aligned}
\]
For the second term, \(\sigma\varphi(x)|y|>1\) implies
\(|y|>(\sigma\|\varphi\|_\infty)^{-1}\).  Hence
\[
\begin{aligned}
\int_{\{|y|\le1\}}|y|\int_0^1
\sigma\varphi(x)
\mathbf 1_{\{\sigma\varphi(x)|y|>1\}}\,\d x\,\lambda(\d y)&
\le \sigma\|\varphi\|_\infty
\int_{\{|y|\le1\}}
\mathbf 1_{\{|y|>(\sigma\|\varphi\|_\infty)^{-1}\}}\,\lambda(\d y)\\
&=\sigma\|\varphi\|_\infty
\lambda\bigl(\{(\sigma\|\varphi\|_\infty)^{-1}<|y|\le1\}\bigr)<\infty.
\end{aligned}
\]
 Combining with two estimates above,
\(\gamma_{\sigma,\varphi}\) is absolutely finite.

Set
\[
    \overline L_t
    :=\sigma\int_0^t\int_0^1
    \varphi(x)\,\Lambda(\d x,\d s)
    -\gamma_{\sigma,\varphi}t,
    \quad t\ge0.
\]
\begin{lemma}\label{lem:spde-projection-comparison}
The process \(\overline L:=(\overline L_t)_{t\ge0}\) is a L\'evy process with triplet
\((0,0,\lambda_{\sigma,\varphi})\), relative to the truncation
\(y\mathbf 1_{\{|y|\le1\}}\).
Let \(z=(z(t))_{0\le t<S}\) be the solution up to its explosion time
\(S\) of the following integral equation
\begin{equation}\label{sde1}
z(t)
=\Phi(0)+\int_0^t
\bigl(-\pi^2 z(s)+f(z(s))+\gamma_{\sigma,\varphi}\bigr)\,\d s
+\overline L_t,
\quad 0\le t<S.
\end{equation}
If \(f\) is convex, then \(\Phi(t)\ge z(t)\) for \(0\le t<T\wedge S\).
\end{lemma}

\begin{proof}

The Poisson representation of \(\Lambda\) and the exponential formula for
Poisson integrals (see, e.g., \cite[Theorem~2.3.7,
(2.9),
Remark in p.~122
and Corollary~2.4.13]{Applebaum04}) give, for every \(\theta\in\R\),
\begin{align*}
\log\E e^{\mathrm i\theta\overline L_t}
&=-\mathrm i\theta\gamma_{\sigma,\varphi}t
+t\int_0^1\int_{\R\setminus\{0\}}
\Bigl(e^{\mathrm i\theta\sigma\varphi(x)y}-1
-\mathrm i\theta\sigma\varphi(x)y\mathbf 1_{\{|y|\le1\}}\Bigr)
\lambda(\d y)\,\d x\\
&=t\int_0^1\int_{\R\setminus\{0\}}
\Bigl(e^{\mathrm i\theta\sigma\varphi(x)y}-1
-\mathrm i\theta\sigma\varphi(x)y
\mathbf 1_{\{|\sigma\varphi(x)y|\le1\}}\Bigr)
\lambda(\d y)\,\d x\\
&=t\int_{\R\setminus\{0\}}
\Bigl(e^{\mathrm i\theta w}-1
-\mathrm i\theta w\mathbf 1_{\{|w|\le1\}}\Bigr)
\lambda_{\sigma,\varphi}(\d w).
\end{align*}
The last line is the L\'evy--Khintchine representation relative to the
truncation \(w\mathbf 1_{\{|w|\le1\}}\), so uniqueness gives the triplet
\((0,0,\lambda_{\sigma,\varphi})\).

Work pathwise on an event of probability one and fix \(t_0<T\wedge S\).
Taking \(\varphi\) as the test function in the weak formulation and using
\(-\varphi''=\pi^2\varphi\), we obtain, for \(0\le t\le t_0\),
\begin{equation}\label{eq:spde-projected-weak-form}
\begin{aligned}
\Phi(t)-\Phi(0)
={}&-\pi^2\int_0^t\Phi(s)\,\d s
+\int_0^t\int_0^1\varphi(x)f(u(s,x))\,\d x\,\d s+\gamma_{\sigma,\varphi}t+\overline L_t.
\end{aligned}
\end{equation}
By the definition of the solution class, \(\Phi\) is c\`adl\`ag on
\([0,T)\).  Since
\(\varphi(x)\,\d x\) is a probability measure on \((0,1)\)
and $f$ is convex,
Jensen's
inequality gives, for almost every \(s\le t_0\),
\begin{equation}\label{eq:spde-projected-jensen}
    \int_0^1\varphi(x)f(u(s,x))\,\d x
    \ge f\left(\int_0^1\varphi(x)u(s,x)\,\d x\right)
    =f(\Phi(s)).
\end{equation}
On the other hand, \eqref{sde1} gives, for \(0\le t\le t_0\),
\[
    z(t)-z(0)
    =\int_0^t    \bigl(-\pi^2 z(s)+f(z(s))+\gamma_{\sigma,\varphi}\bigr)\,\d s
    +\overline L_t.
\]
Subtracting the two identities and using \(z(0)=\Phi(0)\), we find that
\(e:=\Phi-z\) is absolutely continuous on \([0,t_0]\) and
\[
e(t)=\int_0^t\Bigl[-\pi^2e(s)
+\int_0^1\varphi(x)f(u(s,x))\,\d x-f(z(s))\Bigr]\,\d s.
\]
The c\`adl\`ag paths of \(\Phi\) and \(z\) are bounded on \([0,t_0]\).
Let \(C_{t_0}\) be a Lipschitz constant of \(f\) on an interval containing
their ranges, and set \(e^-:=\max\{-e,0\}\).
Since \(e^-(0)=0\), the Sobolev truncation rule
(see \cite[Theorem~4.4(iii)--(iv), p.~153]{EvansGariepy15}) and
\eqref{eq:spde-projected-jensen} give, for \(0\le t\le t_0\),
\[
\begin{aligned}
e^-(t)
&=e^-(t)-e^-(0)\\
&=\int_0^t\mathbf 1_{\{e(s)<0\}}
\Biggl[\pi^2e(s)+f(z(s))-\int_0^1\varphi(x)f(u(s,x))\,\d x\Biggr]\,\d s\\
&\le\int_0^t\mathbf 1_{\{e(s)<0\}}
\Bigl[\pi^2e(s)+f(z(s))-f(\Phi(s))\Bigr]\,\d s\\
&=\int_0^t\mathbf 1_{\{e(s)<0\}}
\Bigl[-\pi^2e^{-}(s)+f(z(s))-f(\Phi(s))\Bigr]\,\d s\\
&\le C_{t_0}\int_0^t e^-(s)\,\d s.
\end{aligned}
\]
In particular,
the integral form of Gronwall's inequality yields
\(e^-=0\) on \([0,t_0]\).  As \(t_0<T\wedge S\) was arbitrary,
\(\Phi(t)\ge z(t)\) for \(0\le t<T\wedge S\).
\end{proof}

Set \(X_t:=-z(t)\) for \(0\le t<S\) and
\(L_t^X:=-\overline L_t\) for \(t\ge0\).  By \eqref{sde1},
\(X:=(X_t)_{0\le t<S}\) solves the following equation up to its explosion time \(S\):
\begin{equation}\label{eq:minus-SDE} 
    X_t
    =X_0-
    \int_0^t\bigl(\pi^2X_s+f(-X_s)+\gamma_{\sigma,\varphi}\bigr)\,\d s
    +L_t^X,
    \quad 0\le t<S,
\end{equation} 
where \(X_0=-\Phi(0)\).  Denote its left and right explosion times by
\(T_-^X\) and \(T_+^X\), respectively.  Then \(S=T_-^X\wedge T_+^X\).
Moreover, \(L^X:=(L^X_t)_{t\ge0}\) has zero drift and a L\'evy measure
\begin{equation}\label{eq:spde-reflected-levy-measure}
    \nu_X(B)
    :=
    \int_0^1\int_{\R\setminus\{0\}}
    \mathbf 1_B\bigl(-\sigma\varphi(x)y\bigr)
    \,\lambda(\d y)\,\d x.
\end{equation}

For \(F\in C^2(\R)\), write
\begin{equation}\label{eq:spde-jump-operator}
\begin{aligned}
\mathcal I_XF(x)
={}&
\int_{\{|y|\le1\}}
\left[F(x+y)-F(x)-yF'(x)\right]\nu_X(\d y)\\
&+
\int_{\{|y|>1\}}
\left[F(x+y)-F(x)\right]\nu_X(\d y).
\end{aligned}
\end{equation}
Define
{\small
\[
\begin{aligned}
\mathcal D_X:=\biggl\{F\in C^2(\R):\quad
&\sup_{x\in K}\int_{\R}
\left|F(x+y)-F(x)-yF'(x)\mathbf 1_{\{|y|\le1\}}\right|
\nu_X(\d y)<\infty\\
&\text{for every compact }K\subset\R\biggr\}.
\end{aligned}
\]
}
For \(F\in\mathcal D_X\), write
\[
    \mathcal L_XF(x)
    :=-b_X(x)F'(x)+\mathcal I_XF(x),
    \quad
    b_X(x):=\pi^2 x+f(-x)+\gamma_{\sigma,\varphi}.
\]
This is the local extended generator of $(X_t)_{0\le t<S}$ on
$\mathcal D_X$: for each $F\in\mathcal D_X$, It\^o's formula gives
the corresponding local martingale identity before $S$.

The following statement can be regarded as a general form of Theorem \ref{cor:spde-explicit-correction}.

\begin{theorem}\label{thm:spde-lyapunov-correction}
Let \(f\) be a nonnegative convex function such that
\[
    \int^\infty\frac{\d r}{f(r)}<\infty,
\]
and assume that \(\lambda((0,\infty))>0\).  Suppose that there exist \(R\in\R\) and a
nonnegative function \(V\in\mathcal D_X\) such that
\[
    \lim_{x\to+\infty}V(x)=\infty,
    \quad
    \mathcal L_XV(x)\le0,\,\, x>R .
\]
Then, for every deterministic \(u_0\in L^2(0,1)\), every local weak
solution \((u,T)\) up to explosion satisfies
\[
    \Pp(T<\infty)=1 .
\]
\end{theorem}

\begin{proof}
By \eqref{eq:spde-reflected-levy-measure},
\[
    \nu_X((-\infty,0))=\lambda((0,\infty))>0,
\]
so Condition~\ref{cond:b-as-coeff}(i) holds.  Moreover, convexity and the
Osgood condition imply \(f(r)/r\to\infty\).  Hence, for all sufficiently
large \(r\),
\[
    h(r):=f(r)-\pi^2 r+\gamma_{\sigma,\varphi}
\]
is positive and nondecreasing, and satisfies \(h(r)\ge f(r)/2\) and so
\[
    \int^\infty\frac{\d r}{h(r)}<\infty.
\]
Since \(b_X(x)=h(-x)\) for all sufficiently negative \(x\), this verifies
Condition~\ref{cond:b-as-coeff}(ii).  The assumed Lyapunov inequality and
Lemma~\ref{lem:b-right-lyapunov-return} give
Condition~\ref{cond:b-right-return}.  Lemma~\ref{lem:b-no-right-explosion-from-return}
therefore gives \(T_+^X=\infty\), so the explosion time of \(z=-X\) is
\(S=T_-^X\).  Theorem~\ref{thm:b-as-explosion} yields \(S<\infty\) almost
surely.  By the definition of \(T_-^X\),
\(\sup_{0\le t<S}z(t)=+\infty\).

It remains to compare \(T\) and \(S\).  On the event \(\{T>S\}\), \(\Phi\)
is a real-valued c\`adl\`ag function on the compact
interval \([0,S]\), and is therefore bounded there.  On the other hand,
the comparison \(\Phi(t)\ge z(t)\) for \(t<S\) and
\(\sup_{0\le t<S}z(t)=+\infty\) give a contradiction.  Consequently,
\(\Pp(T>S)=0\), and hence \(T\le S<\infty\) almost surely.
\end{proof}

Next, we will construct an explicit  Lyapunov function in Theorem \ref{thm:spde-lyapunov-correction} under some mild assumptions.

\begin{lemma}\label{lem:spde-log-lyapunov}
Suppose that
\[
    \int_{(-\infty,-1)}\log(1+|y|)\,\lambda(\d y)<\infty.
\]
Then there exists a nonnegative, nondecreasing function
\(V\in\mathcal D_X\) such that
\[
    \lim_{x\to\infty}V(x)=\infty,
    \quad
    \limsup_{x\to\infty}\mathcal I_XV(x)\le0.
\]
Moreover, \(V'(x)=1/(1+x)\) for \(x\ge1\).
\end{lemma}

\begin{proof}
By \eqref{eq:spde-reflected-levy-measure} and the boundedness of \(\varphi\),
the assumed logarithmic moment implies
\begin{equation}\label{eq:spde-projected-log-moment}
    \int_{(1,\infty)}\log(1+y)\,\nu_X(\d y)<\infty.
\end{equation}
Indeed, marks in \((-\infty,-1)\) are controlled by the assumed logarithmic
moment, while any remaining marks contributing to the integral are bounded
away from zero and hence have finite \(\lambda\)-measure.

Choose a smooth function \(\chi:[0,\infty)\to[0,1]\) such that
\(\chi(x)=0\) for \(0\le x\le1/2\) and \(\chi(x)=1\) for \(x\ge1\), and define
\begin{equation}\label{eq:spde-log-lyapunov-definition}
V(x):=
\begin{cases}
0, & x\le0,\\[2mm]
\displaystyle\int_0^x\frac{\chi(r)}{1+r}\,\d r, & x>0.
\end{cases}
\end{equation}
Since \(\chi=0\) on \([0,1/2]\), the function \(V\) is identically zero
on \((-\infty,1/2]\), and hence \(V\in C^2(\R)\).  For \(x>0\),
\[
    V'(x)=\frac{\chi(x)}{1+x},
    \quad
    V''(x)=\frac{\chi'(x)}{1+x}
             -\frac{\chi(x)}{(1+x)^2},
\]
and therefore
\[
    \|V''\|_\infty\le\|\chi'\|_\infty+1<\infty.
\]
Moreover, \(0\le\chi\le1\) gives \(V'\ge0\).  Since \(V=0\) on
\((-\infty,0]\), the function \(V\) is nonnegative and nondecreasing.
Finally, since \(\chi(x)=1\) for \(x\ge1\),
\begin{equation}\label{eq:spde-log-tail-formulas}
    V'(x)=\frac1{1+x},
    \,\,
    V''(x)=-\frac1{(1+x)^2},
    \,\,
    V(x)=V(1)+\log\frac{1+x}{2},
    \quad x\ge1.
\end{equation}
In particular, \(V(x)\to\infty\) as \(x\to\infty\).

For \(|y|\le1\), Taylor's formula gives
\begin{equation}\label{eq:spde-log-small-jump-bound}
\left|V(x+y)-V(x)-yV'(x)\right|
\le \frac12\|V''\|_\infty y^2.
\end{equation}
Since \(0\le\chi\le1\), \eqref{eq:spde-log-lyapunov-definition} gives
\begin{equation}\label{eq:spde-log-growth-bound}
    0\le V(u)\le \log(1+\max\{0,u\}),
    \quad u\in\R.
\end{equation}
Fix a compact interval \(K\subset\R\) and set
\(M_K:=\max\{0,\sup K\}\).
For \(x\in K\) and \(y>1\),
\[
\left|V(x+y)-V(x)\right|
\le V(x+y)
\le\log(1+M_K+y).
\]
If \(y<-1\), monotonicity and nonnegativity give
\[
\left|V(x+y)-V(x)\right|
=V(x)-V(x+y)
\le V(x).
\]
Consequently, there is \(C_K<\infty\) such that, for all \(x\in K\),
\begin{equation}\label{eq:spde-log-large-jump-bound}
\left|V(x+y)-V(x)\right|
\le
\begin{cases}
C_K\bigl(1+\log(1+y)\bigr), & y>1,\\
C_K, & y<-1.
\end{cases}
\end{equation}
\eqref{eq:spde-log-small-jump-bound} and
\eqref{eq:spde-log-large-jump-bound}, together with
\[
    \int_{\{|y|\le1\}}y^2\,\nu_X(\d y)<\infty,
    \quad
    \nu_X((-\infty,-1))<\infty,
\]
and \eqref{eq:spde-projected-log-moment}, give, for every compact interval
\(K\subset\R\),
\[
\sup_{x\in K}\int_{\R}
\left|V(x+y)-V(x)-yV'(x)\mathbf 1_{\{|y|\le1\}}\right|
\nu_X(\d y)<\infty.
\]
Hence \(V\in\mathcal D_X\).

For \(x>2\), Taylor's formula and
\eqref{eq:spde-log-tail-formulas} give
\begin{equation}\label{eq:spde-log-tail-small-jump}
\left|
\int_{\{|y|\le1\}}
\left[V(x+y)-V(x)-yV'(x)\right]\nu_X(\d y)
\right|
\le
\frac{1}{2x^2}
\int_{\{|y|\le1\}}y^2\,\nu_X(\d y).
\end{equation}
For \(y>1\), \eqref{eq:spde-log-tail-formulas} gives
\begin{equation}\label{eq:spde-log-tail-positive-jump}
0\le V(x+y)-V(x)
=\log\left(1+\frac{y}{1+x}\right)
\le \log(1+y).
\end{equation}
For \(y<-1\), the monotonicity of \(V\) gives
\(V(x+y)-V(x)\le0\).  Hence, \eqref{eq:spde-jump-operator} and
\eqref{eq:spde-log-tail-small-jump} yield
\begin{align*}
\mathcal I_XV(x)
&\le
\frac{1}{2x^2}
\int_{\{|y|\le1\}}y^2\,\nu_X(\d y)
+\int_{(1,\infty)}
\log\left(1+\frac{y}{1+x}\right)\nu_X(\d y).
\end{align*}
The first term tends to zero as \(x\to\infty\).  The second term also tends
to zero by \eqref{eq:spde-log-tail-positive-jump},
\eqref{eq:spde-projected-log-moment}, and dominated convergence.  Therefore
\(\limsup_{x\to\infty}\mathcal I_XV(x)\le0\).
\end{proof}

Now, we are in a position to present the
\begin{proof}[Proof of Theorem $\ref{cor:spde-explicit-correction}$]\,\,
Recall that \(b_X(x)=\pi^2 x+f(-x)+\gamma_{\sigma,\varphi}\).

\noindent\textup{(1)} \textit{Assume that \textup{(i)} is satisfied.}\,\,
Let \(V\) be the function given in
Lemma~\ref{lem:spde-log-lyapunov}.  Then
\[
    \mathcal L_XV(x)=-b_X(x)V'(x)+\mathcal I_XV(x).
\]
For \(x\ge1\), using \(f\ge0\),
\[
    -b_X(x)V'(x)
    =-\frac{\pi^2 x+f(-x)+\gamma_{\sigma,\varphi}}{1+x}
    \le-\frac{\pi^2 x+\gamma_{\sigma,\varphi}}{1+x}.
\]
Therefore,
\[
    \limsup_{x\to\infty}\mathcal L_XV(x)
    \le-\pi^2<0.
\]
Thus \(\mathcal L_XV(x)\le0\) for all sufficiently large \(x\), and
Theorem~\ref{thm:spde-lyapunov-correction} gives the conclusion.

\noindent\textup{(2)} \textit{Assume that \textup{(ii)} is satisfied.}\,\,
Condition \textup{(ii)} and convexity allow us to choose \(A>0\) such that
\(f\) is positive and nonincreasing on \((-\infty,-A]\) and
\(\pi^2 x+\gamma_{\sigma,\varphi}\ge0\) for \(x\ge A\).  Define
\(g_+(x):=f(-x)\) for \(x\ge A\).  Then \(g_+\) is positive and
nondecreasing, and
\begin{equation}\label{eq:spde-corollary-right-tail}
    b_X(x)\ge g_+(x),\quad x\ge A,
    \quad
    \int_A^\infty\frac{\d x}{g_+(x)}
    =\int_{-\infty}^{-A}\frac{\d r}{f(r)}<\infty.
\end{equation}
Thus \eqref{eq:spde-corollary-right-tail} verifies
Condition~\ref{cond:b-lyapunov}.

The assumption \(\int^\infty \d r/f(r)<\infty\) and convexity imply
\(f(r)/r\to\infty\) as \(r\to\infty\).  Choose \(B>0\) such that
\(r\mapsto f(r)-\pi^2 r+\gamma_{\sigma,\varphi}\) is positive and
nondecreasing and
\(f(r)-\pi^2 r+\gamma_{\sigma,\varphi}\ge f(r)/2\) for \(r\ge B\).
Define
\(g_-(x):=f(-x)+\pi^2 x+\gamma_{\sigma,\varphi}\) for \(x\le-B\).
Then \(g_-\) is positive and nonincreasing, and
\begin{equation}\label{eq:spde-corollary-left-tail}
    b_X(x)=g_-(x),\quad x\le-B,
    \quad
    \int_{-\infty}^{-B}\frac{\d x}{g_-(x)}
    =\int_B^\infty
    \frac{\d r}{f(r)-\pi^2 r+\gamma_{\sigma,\varphi}}
    \le2\int_B^\infty\frac{\d r}{f(r)}<\infty.
\end{equation}
Thus \eqref{eq:spde-corollary-left-tail} verifies
Condition~\ref{cond:b-as-coeff}(ii).

By \eqref{eq:spde-reflected-levy-measure},
\(\nu_X((-\infty,0))=\lambda((0,\infty))>0\), and hence
Condition~\ref{cond:b-as-coeff}(i) holds.
By Theorem~\ref{thm:b-main} applied to \eqref{eq:minus-SDE}, there exists $C<\infty$ such that 
$\E T_-^X\le C$ for any $u_0\in L^2(0,1)$. 
Right explosion is excluded by
Lemma~\ref{lem:b-no-right-explosion-from-return}, so $S=T_-^X$.
By Lemma~\ref{lem:spde-projection-comparison}, $\Phi(t)\ge z(t)$ for
$t<T\wedge S$.  On $\{T>S\}$, the c\`adl\`ag function $\Phi$ is
bounded on $[0,S]$, whereas $\sup_{0\le t<S}z(t)=+\infty$, a contradiction.
Thus $T\le S$ almost surely.
Consequently $\E T\le\E S\le C$, proving both
$\Pp(T<\infty)=1$ and~\eqref{eq:spde-general-mean}.
The coefficients of the reflected SDE depend only on $f,\sigma,\lambda$,
so the constant
$C$ 
is independent of $u_0$ and of the chosen weak solution.
\end{proof}

\subsection{A concrete Poisson-driven model}
\label{subsec:spde-poisson-model}

\begin{proof}[Proof of Corollary~$\ref{cor:spde-poisson-maximal}$]
\emph{Step $1$.} We obtain the deterministic solutions to be used between
successive Poisson jumps.
Set $c=\sigma\int_{\{0<|z|\le1\}}z\,\lambda(\dd z)$, which is finite
since $\lambda$ is finite.  For an initial value $\eta\in H^{-1}(0,1)$,
consider
the following partial differential equation on $[0,\infty) \times (0,1)$: 
\begin{equation}\label{eq:spde-deterministic-pde}
\begin{cases}
 \partial_t v=\partial_{xx}v+|v|^p-c,&\quad  t>0,\,\, x\in(0,1),\\
 v(t,0)=v(t,1)=0,&\quad t>0,\\
 v(0,\cdot)=\eta.
\end{cases}
\end{equation}
We apply the semilinear theorem of Matioc and
Walker~\cite[Theorem~1.2]{MatiocWalker23}, checking conditions
(1.7a)--(1.7c) of~\cite{MatiocWalker23} below.
Take $E_0=H^{-1}(0,1)$ and $E_1=H_0^1(0,1)$.
Then $E_1$ embeds continuously and densely in $E_0$, and
$E_{1/2}=[E_0,E_1]_{1/2}=L^2(0,1)$ with equivalent norms;
see~\cite[Section~4.2, (4.4)--(4.5)]{MatiocWalker23}, with Dirichlet
boundary conditions and spatial Lebesgue exponent $2$.

\textbf{Condition (1.7a):}
In the notation of~\cite{MatiocWalker23}, $\mathcal H(E_1,E_0)$ denotes the class
of linear operators with domain $E_1$ that generate strongly continuous
analytic semigroups on $E_0$.
Define $Av=v''$ with domain $E_1$ on $E_0$, where
$\langle Av,\psi\rangle=-\int_0^1v'(x)\psi'(x)\,\dd x$ for $\psi\in E_1$.
For this Dirichlet Laplacian on $H^{-1}(0,1)$ with domain $H_0^1(0,1)$,
\cite[Section~4.2, (4.3)--(4.4)]{MatiocWalker23} gives
$A\in\mathcal H(E_1,E_0)$.
Since $A$ is time-independent, the time-H\"older requirement in (1.7a) is automatic.

\textbf{Condition (1.7b):}
Choose $\alpha=\gamma=0$, $\xi=1/2$ and $q=p$.
Then $0\le\alpha\le\xi\le1$, $0\le\gamma<1$,
$(\gamma,\xi)=(0,1/2)\ne(0,1)$ and $q>1$.  Since $p<2$, the remaining condition holds:
\[
 (\xi-\alpha)q=\frac p2<1=\min\{1,1+\gamma-\alpha\}.
\]

\textbf{Condition (1.7c):}
In~\cite[Theorem~1.2]{MatiocWalker23}, use $F(v)=|v|^p-c$ as the nonlinear map from
$E_\xi=L^2(0,1)$ to $E_\gamma=H^{-1}(0,1)$.
Indeed, the continuous embedding $L^1(0,1)\hookrightarrow H^{-1}(0,1)$
gives $\|g\|_{H^{-1}(0,1)}\le\|g\|_1$.
Since $p<2$ and the interval has length $1$,
$\|F(v)\|_{H^{-1}(0,1)}\le\|F(v)\|_1\le\|v\|_p^p+|c|\le\|v\|_2^p+|c|$.
Moreover, the inequality
\[
 \bigl||a|^p-|b|^p\bigr|
 \le C_p\bigl(|a|^{p-1}+|b|^{p-1}\bigr)|a-b|
\]
and H\"older's inequality with exponents $p/(p-1)$ and $p$ yield
\[
 \begin{aligned}
 \|F(v)-F(w)\|_{H^{-1}(0,1)}
 &\le\big\||v|^p-|w|^p\big\|_1\\
 &\le C_p\bigl(\|v\|_p^{p-1}+\|w\|_p^{p-1}\bigr)\|v-w\|_p\\
 &\le C_p\bigl(\|v\|_2^{p-1}+\|w\|_2^{p-1}\bigr)\|v-w\|_2.
 \end{aligned}
\]
This proves (1.7c).

Applying~\cite[Theorem~1.2(i)]{MatiocWalker23} and its time-weighted
estimate with $p<2$, we obtain for each initial value $\eta\in H^{-1}(0,1)$
a solution $v^\eta$ on $[0,\zeta(\eta))$,
where $\zeta(\eta)\in(0,\infty]$, such that
\[
 v^\eta\in C([0,a];H^{-1}(0,1))\cap L^p((0,a);L^2(0,1)),
 \quad 0<a<\zeta(\eta).
\]
Indeed, \cite[Theorem~1.2(i)]{MatiocWalker23} yields for any $\mu>1/2$, $\lim_{t\rightarrow 0}t^{\mu}\|v^{\eta}(t)\|_{L^2(0,1)}=0$.
Since $1<p<2$, we can take $\mu>1/2$ so that $p\mu <1$, which implies  $v^\eta\in L^p((0,a);L^2(0,1))$. In particular, 
$v^\eta$ satisfies the weak formulation
of~\eqref{eq:spde-deterministic-pde}.
If $\zeta(\eta)<\infty$, \cite[Theorem~1.2(ii)]{MatiocWalker23} gives
\[
\sup_{0\le t<\zeta(\eta)}\|v^\eta(t)\|_{H^{-1}(0,1)}=\infty.
\]

\emph{Step $2$.} We construct the process by restarting these solutions
at the Poisson jump times.
The finite-activity noise has the decomposition
\begin{equation}\label{eq:spde-finite-noise}
 \Lambda(\dd x,\dd t)
 =\int_{\R\setminus\{0\}}z\,\mu(\dd x,\dd z,\dd t)
  -\left(\int_{\{0<|z|\le1\}}z\,\lambda(\dd z)\right)\dd x\,\dd t.
\end{equation}
Enumerate its atoms as $(\xi_n,Z_n,\tau_n)$, where
$0<\tau_1<\tau_2<\cdots\uparrow\infty$ almost surely.
Put $\tau_0=0$ and $\eta_0=u_0$.
After the $n$th jump, run $v^{\eta_n}(t-\tau_n)$ with the same $c$.
If this segment ends at or before $\tau_{n+1}$, stop; otherwise set
\begin{equation}\label{eq:spde-poisson-restart}
 \eta_{n+1}=v^{\eta_n}(\tau_{n+1}-\tau_n)
                  +\sigma Z_{n+1}\delta_{\xi_{n+1}}.
\end{equation}
Since $\delta_{\xi_{n+1}}\in H^{-1}(0,1)$, we have
$\eta_{n+1}\in H^{-1}(0,1)$ and can restart the deterministic solution
after the $(n+1)$st jump. 
Denote the resulting process by $u$ and its endpoint by $T\in(0,\infty]$.

\emph{Step $3$.} We verify that $(u,T)$ is a local weak solution up to
explosion and then apply Theorem~\ref{cor:spde-explicit-correction}.
The process has c\`adl\`ag
$H^{-1}(0,1)$ 
paths.
For each $t<T$, only finitely many deterministic segments occur on $[0,t]$.
Summing their weak formulations of~\eqref{eq:spde-deterministic-pde}
and the jump increments in~\eqref{eq:spde-poisson-restart}, and
using~\eqref{eq:spde-finite-noise}, gives the weak formulation
of~\eqref{e:SPDE}.

It remains to verify~\eqref{eq:spde-norm-explosion}.
On the event $\{T<\infty\}$, only finitely many jumps have occurred,
so there is a final deterministic segment, with
$T=\tau_n+\zeta(\eta_n)\le\tau_{n+1}$.
The deterministic blow-up criterion gives
$\sup_{\tau_n\le t<T}\|u(t)\|_{H^{-1}(0,1)}=\infty$,
which verifies~\eqref{eq:spde-norm-explosion}.
Thus $(u,T)$ is a local weak solution up to explosion.

Finally, the reaction function $r\mapsto|r|^p$ and $\lambda((0,\infty))>0$ verify the
hypotheses of Theorem~\ref{cor:spde-explicit-correction}.
Applying that theorem to $(u,T)$ gives
\eqref{eq:spde-poisson-mean}, uniformly in $u_0$.
In particular, $\Pp(T<\infty)=1$.
\end{proof}

\bigskip

\noindent {\bf Acknowledgements.}\,\,
The research of Pei-Sen Li is supported by the National Natural Science Foundation of China (No.\ 12271029).
The research of Yuichi Shiozawa is supported by JSPS KAKENHI Grant Number JP23K25773.
The research of Jian Wang is supported by the NSF of China the National Key R\&D Program of China (2022YFA1006003) and the National Natural Science
Foundation of China (Nos. 12225104 and 12531007).

\endgroup

\end{document}